\documentclass[10pt,reqno]{amsart}

\usepackage{amsmath}
\usepackage{amssymb}
\usepackage{amsthm}
\usepackage{xcolor}
\usepackage{graphicx}
\usepackage{url}
\usepackage{multirow}

\numberwithin{equation}{section}

\newtheorem{theorem}{Theorem}[section]

\newtheorem{lemma}[theorem]{Lemma}
\newtheorem{prop}[theorem]{Proposition}
\newtheorem{conj}[theorem]{Conjecture}
\theoremstyle{definition}

\theoremstyle{remark}
\newtheorem{rem}[theorem]{Remark}

\newcommand{\qbinom}[2]{\genfrac{[}{]}{0pt}{}{#1}{#2}}
\newcommand{\floor}[1]{\left\lfloor #1 \right\rfloor}

\renewcommand{\Re}{\operatorname{Re}}
\renewcommand{\Im}{\operatorname{Im}}

\DeclareMathOperator{\erfc}{erfc}
\newcommand{\R}{\mathbb{R}}
\newcommand{\Z}{\mathbb{Z}}
\newcommand{\C}{\mathbb{C}}

\begin{document}

\title{An analytic proof of the Borwein Conjecture}
\author{Chen Wang}
\address[Wang]{University of Vienna}
\email{chen.wang@univie.ac.at}
\thanks{This work is supported by the Austrian Science Fund (FWF) grant SFB F50 (F5009-N15).}

\begin{abstract}
We provide a proof of the Borwein Conjecture using analytic methods.
\end{abstract}

\maketitle

\section{Introduction}

In 1990, Peter Borwein observed that the coefficients of the polynomials
\[
\prod_{i=1}^{n}(1-q^{3i-2})(1-q^{3i-1})
\]
have a repeating sign pattern of $+--$. A more formalized version appears in an
1995 paper by Andrews \cite{MR1395410}. Here, and in the sequel, we use the standard notation for $q$-shifted factorials,
\begin{align*}
(a;q)_n&=(1-a)(1-aq)\cdots(1-aq^{n-1}), \text{ for }n\geq1,\\
(a;q)_0&=1.
\end{align*}
\begin{conj}[\sc P. Borwein]\label{cjBorwein}
Let the polynomials $A_n(q)$, $B_n(q)$ and $C_n(q)$ be defined by the relationship
\begin{equation}\label{eqConjecture}
\frac{(q;q)_{3n}}{(q^3;q^3)_n}=A_n(q^3)-qB_n(q^3)-q^2C_n(q^3).
\end{equation}
 Then these polynomials have non-negative coefficients.
\end{conj}
This statement is known as the \emph{Borwein Conjecture}.

There have been many attempts to prove the Borwein Conjecture. Moreover, we find several variations and generalizations in the literature, see \cite{MR1395410,MR2140441,MR1392489,MR1874535,MR2009544,MR2220659}, sometimes also conjecturally, sometimes with full or partial proofs. However, none of the proved variations and generalizations cover the original conjecture, Conjecture \ref{cjBorwein}.
It is fair to say that so far essentially two methods have been tried: bijective methods---such as in \cite{MR1392489,MR1660081}, and basic hypergeometric methods---such as in \cite{MR1395410,MR2140441,MR2009544}. Surprisingly though, it seems that nobody made an asymptotic attack on the conjecture. This may have to do with the fact that the ``canonical" formulas for $A_n(q)$, $B_n(q)$ and $C_n(q)$, namely \eqref{eqExpansionOldA}--\eqref{eqExpansionOldC}, are entirely unsuitable for asymptotic approximation, see the corresponding remarks in Section~\ref{seDiscuss}. Nonetheless, it turns out that there are formulas for $A_n(q)$, $B_n(q)$ and $C_n(q)$ that are amenable to asymptotics, which appear already in Andrews' paper \cite{MR1395410}, where the original conjecture appears for the first time in print.

\begin{theorem}[\sc Andrews, {\cite[Theorem 4.1]{MR1395410}}]\label{thAndrewsExpansion}
Let $A_n(q)$, $B_n(q)$ and $C_n(q)$ be defined as in \eqref{eqConjecture}. Then we have the expansions
\begin{align}
A_n(q)&=\sum_{j=0}^{n/3}\frac{q^{3j^2}(1-q^{2n})(q^3;q^3)_{n-j-1}(q;q)_{3j}}{(q;q)_{n-3j}(q^3;q^3)_{2j}(q^3;q^3)_{j}}, \label{eqExpansionA}\\
B_n(q)&=\sum_{j=0}^{(n-1)/3}\frac{q^{3j^2+3j}(1-q^{3j+2}+q^{n+1}-q^{n+3j+2})(q^3;q^3)_{n-j-1}(q;q)_{3j}}{(q;q)_{n-3j-1}(q^3;q^3)_{2j+1}(q^3;q^3)_{j}}, \label{eqExpansionB} \\
C_n(q)&=\sum_{j=0}^{(n-1)/3}\frac{q^{3j^2+3j}(1-q^{3j+1}+q^{n}-q^{n+3j+2})(q^3;q^3)_{n-j-1}(q;q)_{3j}}{(q;q)_{n-3j-1}(q^3;q^3)_{2j+1}(q^3;q^3)_{j}}. \label{eqExpansionC}
\end{align}
\end{theorem}

As a matter of fact, after discussing these formulas briefly, Andrews says in \cite{MR1395410} that {\it``it might be possible to prove that $A_n(q)$ has positive coefficients by establishing sufficiently tight bounds on the coefficients that arise term-by-term in (4.5)"}, where Andrews' (4.5) is our \eqref{eqExpansionA}.

In the present paper, we follow Andrews' advice. Our main discovery is that, in the sums \eqref{eqExpansionA}--\eqref{eqExpansionC}, the first term, i.e, the term for $j=0$, dominates all other terms. This makes these expressions superior to all other known expressions for the purpose of asymptotic estimations. We use analytic methods to bound the coefficients of $A_n(q)$, $B_n(q)$ and $C_n(q)$ away from $0$ by expressing the coefficients as certain contour integrals and estimating these integrals. Section~\ref{seOutline} contains the basic setting of our proof: it is explained how to break the contour integrals into a positive-valued main part and four error terms, thus reducing the Borwein Conjecture to the problem of obtaining sufficiently good upper bounds on the error terms.

After establishing some basic facts and fixing some parameters in Sections~\ref{sePrelim}--\ref{seCutoff}, we derive upper bounds for each of the error terms in Sections~\ref{seEps0}--\ref{seEps2}, which leads to a proof of the Borwein Conjecture for all $n>7000$ in Section~\ref{seMain}.
Some auxiliary results of technical nature are stated and proved separately
in an appendix.
The cases where $0\leq n\leq7000$ are directly verified by a computer calculation, see Section~\ref{seVerify}. We conclude our paper with Section~\ref{seDiscuss}, in which we recall in more detail the earlier mentioned variations and generalizations, and where we also comment on possible further implications of our analytic approach.

\section{An outline of the proof}\label{seOutline}
In this section, we provide a brief outline of our proof of the Borwein Conjecture.

\medskip
First, we claim that non-negativity of the coefficients of $B_n(q)$ already implies the complete Borwein Conjecture. Indeed, we have
\begin{equation}\label{eqBCRelation}
C_n(q)=q^{\deg B_n}B_n(1/q),
\end{equation}
which proves the non-negativity of the coefficients of $C_n(q)$ given the non-negativity of the coefficients of $B_n(q)$. On the other hand, the elementary recursive formula \cite[Eq. (3.3)]{MR1395410}
\begin{equation}\label{eqARecursion}
A_n(q)=(1+q^{2n-1})A_{n-1}(q)+q^n(B_{n-1}(q)+C_{n-1}(q))
\end{equation}
allows us to get the non-negativity of the coefficients of $A_n(q)$ inductively from the non-negativity of the coefficients of $B_n(q)$ (and $C_n(q)$). Therefore, from now on, we will concentrate on $B_n(q)$.

In Section~\ref{sePrelim}, we start by writing (see \eqref{eq:Bnj})
$$
B_n(q)=\sum_{j=0}^{(n-1)/3}B_{n,j}(q),
$$
where $B_{n,j}(q)$ is the $j$-th summand in the expansion \eqref{eqExpansionB}.
We then decompose $B_{n,j}(q)$ into the sum of two simpler polynomials, namely $D_{n,j}(q)$ and $E_{n,j}(q)$, see \eqref{eqExpansionD}, \eqref{eqExpansionE},
and \eqref{eq:DEnj}, so that
$$
B_{n,j}(q)=q(1+q^n)D_{n,j}(q)+E_{n,j}(q).
$$
The background of this decomposition is that the polynomials $D_{n,j}(q)$ and $E_{n,j}(q)$ are simpler to handle asymptotically. By summing over all~$j$,
we define
\begin{align*}
D_n(q)&:=\sum_{j=0}^{(n-1)/3}D_{n,j}(q),&
E_n(q)&:=\sum_{j=0}^{(n-1)/3}E_{n,j}(q),
\end{align*}
so that
$$
B_{n}(q)=q(1+q^n)D_{n}(q)+E_{n}(q).
$$
In particular, this decomposition shows that, to prove the non-negativity of the coefficients of $B_n(q)$, it suffices to prove the non-negativity of the coefficients of $D_n(q)$ and $E_n(q)$ separately.
Some elementary properties about $D_n(q)$ and $E_n(q)$ are collected in Lemma~\ref{leBasics}. In particular, it turns out that $D_n(q)$ is a palindromic polynomial, that is, $D_n(q)=q^{\deg D_n}D_n(1/q)$, while $E_n(q)$ is not. The latter is the reason that, in the subsequent discussion, we also need the reciprocal polynomial of $E_n(q)$, that is, $F_n(q)=q^{\deg E_n}E_n(1/q)$.

The contents of Section~\ref{seInfinite} is a proof of non-negativity
of the coefficients of $q^m$ in $D_n(q),E_n(q)$ and $F_n(q)$ for
$0\leq m< n$. It relies on results of Andrews in \cite{MR1395410}
and on a positivity result of Berkovich and Garvan from \cite{MR2177487}.
Thus, what remains to show, is non-negativity of the
coefficients of $q^m$ in $D_n(q)$ for $n\le m\le(\deg D_n)/2$, and an analogous result for $E_n(q)$ and for $F_n(q)$.

For notational simplicity, we will use the notations $P_n(q)$ and $P_{n,j}(q)$ throughout this paper to refer to multiple families of polynomials. For example, a proposition that is true for $P_n(q)$ for $P\in\{D,E,F\}$ means the proposition is true for all three families of polynomials $D_n(q)$, $E_n(q)$ and $F_n(q)$. We will also use the standard notation $[q^m]P_n(q)$ to represent the coefficient of $q^m$ in the polynomial $P_n(q)$.

Using Cauchy's integral formula, the coefficient $[q^m]P_{n}(q)$ can be represented as the integral
\[
\frac{1}{2\pi i}\int_{\Gamma} P_{n}(q)\frac{dq}{q^{m+1}},
\]
where $\Gamma$ is any contour about~$0$ with winding number 1. We will choose $\Gamma$ as a circle centred at $0$ with radius $r$ for some $r\in\R^+$, so that the integral becomes
\begin{equation}
[q^m]P_{n}(q)=\frac{r^{-m}}{2\pi}\int_{-\pi}^{\pi}P_{n}\left(re^{i\theta}\right)e^{-im\theta}\,d\theta.\label{eqIntRep}
\end{equation}
The exact choice of $r$ is related to the \emph{saddle point\/} of $q^{-m}P_{n,0}(q)$. We will elaborate on this in Section~\ref{seLocate}.
The appropriate choice for~$r$ is a value smaller than~$1$ but close to~$1$,
see Lemma~\ref{leRadiusBound}.

We use the expansions $P_n(q)=\sum_j P_{n,j}(q)$ to write the integral \eqref{eqIntRep} as
\begin{equation}
[q^m]P_n(q)=\sum_{j\geq0}\frac{r^{-m}}{2\pi}\int_{-\pi}^{\pi}P_{n,j}\left(re^{i\theta}\right)e^{-im\theta}\,d\theta. \label{eqIntRep2}
\end{equation}

Figure \ref{fiCircle} illustrates the typical behaviour of $|D_{n,j}\left(re^{i\theta}\right)|$ on the circle $\partial B(0,r)=\{z\in\C\mid\,|z|=r\}$. In particular, we can observe the following general features in the graph:
\begin{itemize}
\item the terms with smaller $j$ have a central peak at $\theta=0$;
\item the central peak of $|P_{n,j}\left(re^{i\theta}\right)|$ for small $j$ looks like a translated-down version of the central peak for $|P_{n,0}\left(re^{i\theta}\right)|$. Since Figure~\ref{fiCircle} is on a logarithmic scale, this suggests that the magnitude $|P_{n,j}\left(re^{i\theta}\right)|$ could be controlled by a constant factor times $|P_{n,0}\left(re^{i\theta}\right)|$ in a neighbourhood of $\theta=0$;
\item for these terms, the parts outside the small neighbourhood of $\theta=0$ are very small compared to the peak value;
\item when $j$ becomes larger, the central peak disappears. However, the graph suggests that all of $|P_{n,j}\left(re^{i\theta}\right)|$ (represented by the red curve in the graph) is at a lower location in the graph, indicating that $|P_{n,j}\left(re^{i\theta}\right)|$ could be controlled by a relatively small constant if $j$ is large.
\end{itemize}

\begin{figure}
\includegraphics[width=0.8\textwidth]{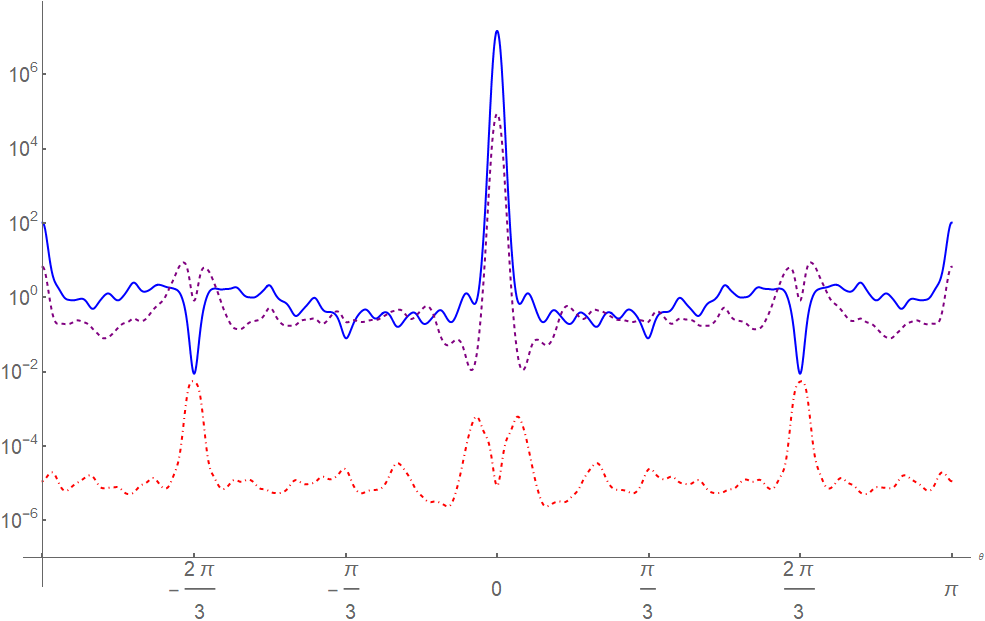}
\caption{Modulus of $D_{36,0}(0.95e^{i\theta})$ (blue), of $D_{36,2}(0.95e^{i\theta})$ (purple, dashed), and of $D_{36,8}(0.95e^{i\theta})$ (red, dot-dashed). The vertical axis has a logarithmic scale.}\label{fiCircle}
\end{figure}

Therefore, based on these heuristics, we choose two cut-offs $j_0$ and $\theta_0$ (to be determined in \eqref{eqCutoffTheta0} and \eqref{eqCutoffJ0}), and distinguish the following parts of the integrands $P_{n,j}\left(re^{i\theta}\right)e^{-im\theta}$, for $0\leq j\leq (n-1)/3$:
\begin{itemize}
\item The term \emph{primary peak} refers to the part where $j=0$ and $|\theta|\le\theta_0$.
\item The term \emph{secondary peaks} refers to the parts where $1\leq j\leq j_0$ and $|\theta|\le\theta_0$.
\item The term \emph{tails} refers to the parts where $0\leq j\leq j_0$ and $\theta_0<|\theta|\leq\pi$.
\item Finally, the term \emph{remainders} refers to the parts where $j>j_0$.
\end{itemize}

Naturally, the integral \eqref{eqIntRep2} can be divided into four sub-integrals corresponding to the four parts above.

For all $P\in\{D,E,F\}$, we make the following observations concerning the four sub-integrals:
\begin{itemize}
\item The primary peak can be approximated by a Gau\ss ian integral. More specifically, if we define
\begin{equation}\label{eqDefG}
g_P(n,r)=\left.-\frac{\partial^2}{\partial\theta^2}\log P_{n,0}(re^{i\theta})\right|_{\theta=0},
\end{equation}
then we have
\begin{equation}\label{eqPart1}%\label{eqDefEps0}
\int_{-\theta_0}^{\theta_0}P_{n,0}(re^{i\theta})e^{-im\theta}\,d\theta\approx P_{n,0}(r)\frac{\sqrt{2\pi}}{\sqrt{g_P(n,r)}}
\end{equation}
for large $n$.
\item The secondary peaks will be bounded above by a constant times the primary peak. We make the argument
\begin{align}
\left|\sum_{j=1}^{j_0}\int_{-\theta_0}^{\theta_0}P_{n,j}(re^{i\theta})e^{-im\theta}\,d\theta\right|&\leq \sum_{j=1}^{j_0}\int_{-\theta_0}^{\theta_0}\left|P_{n,j}(re^{i\theta})\right|\,d\theta\nonumber \\
&\leq\sum_{j=1}^{j_0}\left(\sup_{|\theta|\le\theta_0}\left|\frac{P_{n,j}\left(re^{i\theta}\right)}{P_{n,0}\left(re^{i\theta}\right)}\right|\right)\int_{-\theta_0}^{\theta_0}|P_{n,0}(re^{i\theta})|\,d\theta\label{eqPart2}.
\end{align}
\item The tails will be estimated relative to its corresponding (primary or secondary) peak. More specifically, for $P\in \{D,E,F\}$, we will construct families of polynomials $\tilde{P}_{n,j}(r)$ with non-negative coefficients (see the paragraph before \eqref{eqDTildeRatio1}), acting as uniform upper bounds for $|P_{n,j}(re^{i\theta})|$ over the circle $\partial B(0,r)=\{z\in\C\mid\,|z|=r\}$, satisfying the relations
\begin{align}
\tilde{P}_{n,0}(r)&=P_{n,0}(r),\nonumber \\
\tilde{P}_{n,j}(|q|)&\geq\left|P_{n,j}(q)\right|, \label{eqPTildeCond}
\end{align}
for all $q\in\C$ and all $r\in\R^+$.

With the help of $\tilde{P}_{n,j}(r)$, the tail integrals can be bounded above by
\begin{multline}
\kern1cm
\left|\sum_{j=0}^{j_0}\int^{2\pi-\theta_0}_{\theta_0}P_{n,j}(re^{i\theta})e^{-im\theta}\,d\theta\right|
\leq \sum_{j=0}^{j_0} \tilde{P}_{n,j}(r)\int^{2\pi-\theta_0}_{\theta_0}\left|\frac{P_{n,j}(re^{i\theta})}{\tilde{P}_{n,j}(r)}\right|\,d\theta \\
\leq \left(\sum_{j=0}^{j_0}\frac{\tilde{P}_{n,j}(r)}{\tilde{P}_{n,0}(r)}\right)\left(\int^{2\pi-\theta_0}_{\theta_0}\sup_{0\leq j\leq j_0}\left|\frac{P_{n,j}(re^{i\theta})}{\tilde{P}_{n,j}(r)}\right|\,d\theta\right)\times P_{n,0}(r).\label{eqPart3}
\end{multline}

\item The remainder will be directly controlled by the upper bounds $\tilde{P}_{n,j}(r)$. Namely, by \eqref{eqPTildeCond}, we have
\begin{align}
\sum_{j>j_0}\left|\int^{-\pi}_{\pi}P_{n,j}(re^{i\theta})e^{-im\theta}\,d\theta\right|&\leq \left(2\pi \sum_{j>j_0}\frac{\tilde{P}_{n,j}(r)}{\tilde{P}_{n,0}(r)}\right)\times P_{n,0}(r). \label{eqPart4}
\end{align}
\end{itemize}

\medskip
Our next step is to estimate the relative error in the approximation \eqref{eqPart1}, and to bound the other parts of the integral relative to the (presumably) dominating part $P_{n,0}(r)\frac{\sqrt{2\pi}}{\sqrt{g_P(n,r)}}$. Based on
\eqref{eqPart1} and the inequalities \eqref{eqPart2}, \eqref{eqPart3} and \eqref{eqPart4}, we make the following definitions of the relative errors:
\begin{align}
\label{eq:ep0}
\epsilon_{0,P}(n,m,r)&:=\left|\frac{\sqrt{g_P(n,r)}}{\sqrt{2\pi}P_{n,0}(r)}\int_{-\theta_0}^{\theta_0}P_{n,0}(re^{i\theta})e^{-im\theta}\,d\theta-1\right|,\\
\label{eq:ep1}
\epsilon_{1,P}(n,r)&:=\left(\sum_{j=1}^{j_0}\sup_{|\theta|<\theta_0}\left|\frac{P_{n,j}\left(re^{i\theta}\right)}{P_{n,0}\left(re^{i\theta}\right)}\right|\right)\left(
\frac{\sqrt{g_P(n,r)}}{\sqrt{2\pi}}\int_{-\theta_0}^{\theta_0}\left|\frac{P_{n,0}(re^{i\theta})}{P_{n,0}(r)}\right|\,d\theta\right),\\
\label{eq:ep2}
\epsilon_{2,P}(n,r)&:=\frac{\sqrt{g_P(n,r)}}{\sqrt{2\pi}}\left(\sum_{j=0}^{j_0}\frac{\tilde{P}_{n,j}(r)}{\tilde{P}_{n,0}(r)}\right)\left(\int^{2\pi-\theta_0}_{\theta_0}\sup_{0\leq j\leq j_0}\left|\frac{P_{n,j}(re^{i\theta})}{\tilde{P}_{n,j}(r)}\right|\,d\theta\right),\\
\label{eq:ep3}
\epsilon_{3,P}(n,r)&:=\sqrt{2\pi g_P(n,r)}\sum_{j>j_0}\frac{\tilde{P}_{n,j}(r)}{\tilde{P}_{n,0}(r)}.
\end{align}
It should be noted that only the first of these, $\epsilon_{0,P}(n,m,r)$,
depends on~$m$, the parameter which keeps track of the monomial $q^m$
of which we are taking the coefficient in $P_n(q)$.

These definitions, along with the integral representation \eqref{eqIntRep2} and the inequalities \eqref{eqPart2}, \eqref{eqPart3} and \eqref{eqPart4}, imply that
\begin{equation}
[q^m]P_n(q)\geq\frac{P_{n,0}(r)}{r^m\sqrt{2\pi g_P(n,r)}}\left(1-\epsilon_{0,P}(n,m,r)-\epsilon_{1,P}(n,r)-\epsilon_{2,P}(n,r)-\epsilon_{3,P}(n,r)\right).
\end{equation}
Therefore, once we have sufficiently good bounds on all these error terms so that their sum is smaller than 1, we can conclude that $[q^m]P_n(q)$ is indeed positive.

The primary peak error $\epsilon_{0,P}(n,m,r)$ is estimated in Section~\ref{seEps0}, the secondary peaks $\epsilon_{1,P}(n,r)$ are bounded in Section~\ref{seEps1}, Section~\ref{seEps3} is devoted to bounding the remainders $\epsilon_{3,P}(n,r)$, and finally Section~\ref{seEps2} treats the tails $\epsilon_{2,P}(n,r)$. All these estimations are valid for $n>7000$ and $n\le m\le (\deg D_n)/2$ respectively $n\le m\le (\deg E_n)/2=(\deg F_n)/2$, and their combination shows that the Borwein Conjecture holds for $n>7000$, see Theorem~\ref{thMain} in Section~\ref{seMain}. The cases where $n\le 7000$ are disposed of by a (lengthy) computer calculation, the principles of which are explained in Section~\ref{seVerify}.

\section{Decomposing $B_n(q)$}\label{sePrelim}
As we already explained in the introduction, the starting point of our proof of the Borwein Conjecture is
Theorem~\ref{thAndrewsExpansion}, which provides certain expansions of the polynomials $A_n(q)$, $B_n(q)$ and $C_n(q)$.
Based on the expansion \eqref{eqExpansionB}, we define the family of polynomials $B_{n,j}(q)$ to be the summands in that expansion, so that
\begin{equation} \label{eq:Bnj}
B_n(q)=\sum_{j=0}^{(n-1)/3}B_{n,j}(q).
\end{equation}

%In Section \ref{seEps1}, we will make use bounds for the quotients of the form $\frac{P_{n,j}(q)}{P_{n,j-1}(q)}$ etc.~
The factor $(1-q^{3j+2}+q^{n+1}-q^{n+3j+2})$ in $B_{n,j}(q)$ turns out to be
inconvenient, since our strategy is to bound quotients
$B_{n,j}(q)/B_{n,j-1}(q)$ of successive terms. Therefore, we decompose it as
\begin{equation*}
1-q^{3j+2}+q^{n+1}-q^{n+3j+2}=(1-q)+q(1+q^n)(1-q^{3j+1}).
\end{equation*}
This decomposition naturally extends to the family of polynomials $B_{n,j}(q)$ via the following definitions:
\begin{align}
D_{n,j}(q)&:=\frac{(1-q^{3j+1})}{1-q^{3j+2}+q^{n+1}-q^{n+3j+2}}B_{n,j}(q)\nonumber \\
&=\frac{q^{3j^2+3j}(q^3;q^3)_{n-j-1}(q;q)_{3j+1}}{(q;q)_{n-3j-1}(q^3;q^3)_{2j+1}(q^3;q^3)_{j}},\label{eqExpansionD}\\
E_{n,j}(q)&:=\frac{1-q}{1-q^{3j+2}+q^{n+1}-q^{n+3j+2}}B_{n,j}(q)\nonumber \\
&=\frac{q^{3j^2+3j}(1-q)(q^3;q^3)_{n-j-1}(q;q)_{3j}}{(q;q)_{n-3j-1}(q^3;q^3)_{2j+1}(q^3;q^3)_{j}},\label{eqExpansionE}
\end{align}
so that
\begin{equation} \label{eq:DEnj}
B_{n,j}(q)=q(1+q^n)D_{n,j}(q)+E_{n,j}(q).
\end{equation}
By summing over all~$j$,
we define
\begin{align}
D_n(q)&:=\sum_{j=0}^{(n-1)/3}D_{n,j}(q),&
E_n(q)&:=\sum_{j=0}^{(n-1)/3}E_{n,j}(q),
\label{eq:DEsumdef}
\end{align}
so that
\begin{equation}
\label{eqDecompB}
B_{n}(q)=q(1+q^n)D_{n}(q)+E_{n}(q).
\end{equation}

As we already indicated in the previous section,
our estimations of the error terms
$\epsilon_{0,P}(n,m,r)$ for $P\in\{D,E\}$ are only valid
for $m\le(\deg P_n)/2$, that is, only for ``half of the coefficients",
see Section~\ref{seLocate}, and in particular
Lemma~\ref{leRadiusBound} to which we shall constantly refer.
While this is fine for $D_n(q)$ --- since $D_n(q)$ is palindromic,
proving bounds for the first half of the coefficients automatically
means to also have proved analogous bounds for ``the other half" --- this
is a problem for $E_n(q)$ which is not palindromic. Here, we need to
consider the reciprocal polynomial of $E_n(q)$, that is,
$F_n(q):=q^{\deg E_n}E_n(1/q)$, and also prove estimations for
$\epsilon_{i,F}$ as defined in \eqref{eq:ep0}--\eqref{eq:ep3}. It is a routine calculation from
\eqref{eq:DEsumdef} that with
\begin{equation}\label{eqExpansionF}
F_{n,j}(q):=q^{3j}E_{n,j}(q)=\frac{q^{3j^2+6j}(1-q)(q^3;q^3)_{n-j-1}(q;q)_{3j}}{(q;q)_{n-3j-1}(q^3;q^3)_{2j+1}(q^3;q^3)_{j}}
\end{equation}
we have
\begin{equation} \label{eq:Fsum}
F_n(q)=\sum_{j=0}^{(n-1)/3}F_{n,j}(q).
\end{equation}

\begin{rem} \label{rem1}
It is not hard to see that the functions $D_{n,j}(q)$, $E_{n,j}(q)$
and $F_{n,j}(q)$, as defined above, are actually polynomials for all $j$ with
$0\leq j\leq\floor{(n-2)/3}$. (For a proof of this fact, see the factorizations \eqref{eqDnjFactor1}--\eqref{eqEnjFactor2} and the related discussions in Section~\ref{seTilde}.) However, in the special case
$n\equiv1\pmod{3}$ and $j=(n-1)/3$, \eqref{eqExpansionD},
\eqref{eqExpansionE} and \eqref{eqExpansionF} fail to give
polynomials. Thus, we restrict the domain of the definitions
\eqref{eqExpansionD}, \eqref{eqExpansionE} and\eqref{eqExpansionF} to $0\leq j\leq\floor{(n-2)/3}$, and make alternative definitions in the ``borderline case":
\begin{align}
D_{3j+1,j}(q)&:=0,\label{eqExpansionD1}\\
E_{3j+1,j}(q)&:=B_{3j+1,j}(q)=\frac{q^{3j^2+3j}(q;q)_{3j}}{(q^3;q^3)_{j}},\label{eqExpansionE1}\\
F_{3j+1,j}(q)&:=q^{\deg E_{3j+1}}B_{3j+1,j}(1/q)=\frac{q^{3j^2-2}(q;q)_{3j}}{(q^3;q^3)_{j}}.\label{eqExpansionF1}
\end{align}
It is straightforward to see that,
with these alternate definitions,
and with the sums \eqref{eq:DEsumdef} and \eqref{eq:Fsum}, we still have
\eqref{eqDecompB}.
\end{rem}

We collect some basic facts about these polynomials.
\begin{lemma}\label{leBasics}
For $P\in\{D,E,F\}$, the polynomials $P_n(q)$ and $P_{n,0}(q)$ have the following properties:
\begin{itemize}
\item $D_n(q)$ is a palindromic polynomial, while $E_n(q)$ and $F_n(q)$ are reciprocal of each other. Therefore, it suffices to consider the coefficients $[q^m]P_n(q)$ for $0\leq m\leq (\deg P_n)/2$.
\item $\deg D_n(q)=\deg E_n(q)=\deg F_n(q)=n^2-n-2$. Furthermore, we have $\deg P_{n,0}(q)=\deg P_n(q)$ for all $P\in\{D,E,F\}$.
\item The $j=0$ terms in the expansions have a nice product form:
\begin{equation}
D_{n,0}(q)=E_{n,0}(q)=F_{n,0}(q)=(1+q^2+q^4)(1+q^3+q^6)\cdots(1+q^{n-1}+q^{2n-2}).\label{eqDEFn0}
\end{equation}
\item The expression \eqref{eqDEFn0} implies the following formula for $g_P(n,r)$ as defined in \eqref{eqDefG}:
\begin{equation}\label{eqGPFormula}
g_D(n,r)=g_E(n,r)=g_F(n,r)=\sum_{k=2}^{n-1}\frac{k^2r^k(1+4r^k+r^{2k})}{(1+r^k+r^{2k})^2}.
\end{equation}
\end{itemize}
\end{lemma}
%Thus, the coefficients of the $j=0$ terms are trivially positive. %Numerical results suggests that the coefficients of the $j>0$ terms are negligible compared to the corresponding coefficient of the $j=0$ term, with possible exceptions at both ends of the polynomials, as illustrated by Figure \ref{fiCoeffs}.

%Combined with the results of the last subsection, we conclude that it suffices to consider the coefficients $[q^m]P_n(q)$ for $n\leq m\leq \deg P_n/2$.

%\begin{figure}
%\includegraphics[width=0.7\textwidth]{Anj_2.png}
%\caption{Coefficients of $A_{10,0}(q)$ (blue) and $A_{10,1}(q)$ (purple, dashed).}\label{fiCoeffs}
%\end{figure}

%After dealing with the first few coefficients (the ``possible exceptions" above) of $P_{n}(q)$ for $P\in\{A,D,E,F\}$ in the next subsection, we will proceed to show that the above observations indeed hold for all four families $A_n, D_n, E_n$ and $F_n$ for sufficiently large $n$. This will be done by direct estimations of the coefficients in these four families of polynomials. Moreover, the proof exhibits an explicit bound $N_0$ such that Borwein's conjecture holds for all $n>N_0$. The cases where $1\leq n\leq N_0$ will then be checked with explicit computations.

\section{The first $n$ coefficients}\label{seInfinite}
%As Figure \ref{fiCoeffs} suggests, the presumed asymptotic dominance of $P_{n,0}(q)$ fails to hold for coefficients near the end of the polynomials.
In this section, we settle the non-negativity of the first $n$ coefficients of $P_n(q)$ for $P\in\{D,E,F\}$ by considering the $n\to\infty$ limiting case.

To this end, we define
\[
P_{\infty}(q):=\lim_{n\to\infty}P_n(q)
\]
for all $P\in\{B,C,D,E,F\}$.
The following lemma is a direct consequence of
\eqref{eqExpansionD}, \eqref{eqExpansionE} and \eqref{eqExpansionF}.
\begin{lemma}\label{leLimiting}
For all $P\in\{B,C,D,E,F\}$ and all $n\geq0$, we have
\[
P_{n}(q)=P_{\infty}(q)+O(q^n).
\]
\end{lemma}

This lemma says in particular that, for all $P\in\{D,E,F\}$, the
non-negativity of $P_{\infty}(q)$ implies the non-negativity of
$[q^m]P_{n}(q)$ for $m=0,1,\dots,n-1$. The series $P_\infty(q)$,
with $P\in\{D,E,F\}$ have indeed non-negative coefficients as we
are going to show now.

Andrews proved in \cite{MR1395410} that $B_\infty(q)$ and $C_\infty(q)$ have non-negative coefficients. Moreover,
certain Rogers--Ramanujan type identities for modulus 9, first
discovered by Bailey \cite[p.~224]{MR0022816}, give product formulas for these two series.
\begin{lemma}[{\cite[(4.3)--(4.4)]{MR1395410}}]
The power series $B_\infty(q)$ and $C_\infty(q)$ have the closed form expressions
\begin{align*}
B_{\infty}(q)&=\frac{(q^2,q^7,q^9;q^9)_{\infty}}{(q;q)_{\infty}},&
C_{\infty}(q)&=\frac{(q^1,q^8,q^9;q^9)_{\infty}}{(q;q)_{\infty}},
\end{align*}
where we use the short notation
\[(a_1,a_2,\dots,a_k;q)_{\infty}=(a_1;q)_{\infty}(a_2;q)_{\infty}\dots(a_k;q)_{\infty}.\]
\end{lemma}

We proceed to deduce non-negativity results for the power series
$D_\infty(q), E_\infty(q)$ and $F_\infty(q)$ from these forms. By
taking the limit $n\to\infty$ in equations
\eqref{eqExpansionB} and \eqref{eqExpansionC}, and in
\eqref{eqExpansionD}, \eqref{eqExpansionE} and \eqref{eqExpansionF},
we see that
\begin{align*}
D_{\infty}(q)&=C_{\infty}(q),\\
E_{\infty}(q)&=B_{\infty}(q)-qC_{\infty}(q)=(1-q)B_{\infty}(q)+q(B_{\infty}(q)-C_{\infty}(q)),\\
qF_{\infty}(q)&=B_{\infty}(q)-C_{\infty}(q).\\
\end{align*}
An immediate conclusion is that $D_{\infty}(q)$ also has non-negative coefficients. In order to prove analogous results for $E_\infty(q)$ and $F_{\infty}(q)$, it suffices to show that $B_{\infty}(q)-C_{\infty}(q)$ has non-negative coefficients. To prove this claim, we write
\begin{align*}
B_{\infty}(q)-C_{\infty}(q)&=\frac{(q^2,q^7,q^9;q^9)_{\infty}-(q^1,q^8,q^9;q^9)_{\infty}}{(q;q)_{\infty}}\\
&=\frac{1}{(q^3,q^4,q^5,q^6;q^9)}_{\infty}\left(\frac{1}{(q^1,q^8;q^9)_{\infty}}-\frac{1}{(q^2,q^7;q^9)_{\infty}}\right).
\end{align*}
The non-negativity of the last factor follows from the following partition inequality, first proved by Berkovich and Garvan in \cite{MR2177487}, by taking $m=9$, $r=2$ and $L\to\infty$.

\begin{theorem}[\sc Berkovich \& Garvan, {\cite[Theorem~5.3]{MR2177487}}]
Let $L>0$ and $1<r<m-1$. Then the $q$-series
\[
\frac{1}{(q^1,q^{m-1};q^m)_{L}}-\frac{1}{(q^r,q^{m-r};q^m)_{L}}
\]
has non-negative coefficients if and only if $r\nmid m-r$ and $(m-r)\nmid r$.
\end{theorem}

Thus we have proved that $P_{\infty}(q)$ has non-negative coefficients
for $P\in\{D,E,F\}$. Combined with Lemma~\ref{leLimiting}, we have the following result concerning the first $n$ coefficients of $P_n(q)$.
\begin{theorem}\label{eqHeadCoeffs}
For all $n\geq1$, $0\leq m\leq n-1$, and all $P\in\{D,E,F\}$, we have
\[
[q^m]P_n(q)\geq0.
\]
\end{theorem}
%\begin{proof}
%This is an immediate consequence of the non-negativity of $P_{\infty}(q)$ and Lemma~\ref{leLimiting}.
%\end{proof}

%\subsection{Heuristics concerning a contour integral}

%bounds leads to an integer $N_0$ for which the condition of Lemma \ref{leMain} is satisfied.

\section{Locating the saddle point}\label{seLocate}
The results of the last section show that it suffices to consider $[q^m]P_n(q)$ for $m\in[n,(\deg P_n)/2]$.
The purpose of this section is to describe our choice of the radius $r$ in \eqref{eqIntRep}, under the above restriction on~$m$.

Keeping in line with standard practice in analytic combinatorics (cf.\ \cite{MR2483235}), our choice of the radius $r$ will be a saddle point of the function $z\mapsto z^{-m}P_{n,0}(z)$. It turns out that there is a unique saddle point on the positive real axis, and we have very tight bounds on the position of this point under the condition $m\in[n,(\deg P_n)/2]$. These results will be proved in the following lemma. They are vital in our estimations of the error terms $\epsilon_{i,P}$ in Sections~\ref{seEps0}--\ref{seEps2}.

\begin{lemma}\label{leRadiusBound}
For all $P\in\{D,E,F\}$, all integers $n\geq1$, and $m\in(0,\deg P_n)$, the saddle point equation
\begin{equation}\label{eqStationaryPoint}
\frac{d}{dr}\left(r^{-m}P_{n,0}(r)\right)=0
\end{equation}
has a unique solution $r\in\R^+$. Moreover, if $n\leq m\leq (\deg P_n)/2$, then we have $r_0<r\leq1$ where
\begin{equation}\label{eqCutoffR0}
r_0=e^{-\sqrt{\alpha/n}},
\end{equation}
and $\alpha=2/\sqrt{3}$ is the maximum value of the function $x\mapsto\frac{1+2x}{1+x+x^2}$ on $[0,1]$.
\end{lemma}

\begin{proof}
The equation \eqref{eqStationaryPoint} can be transformed into
\[
\frac{r P'_{n,0}(r)}{P_{n,0}(r)}=m.
\]
Let us write $f_{n,P}(r)$ for the left-hand side. From the definition of the polynomials $P_{n,0}$ in \eqref{eqDEFn0}, we have
\begin{equation*}
f_{n,D}(r)=f_{n,E}(r)=f_{n,F}(r)=\sum_{k=2}^{n-1}\frac{k(2r^{2k}+r^k)}{1+r^k+r^{2k}}.
\end{equation*}
These functions attain the special values
\begin{align} \label{eq:fnP(1)}
f_{n,P}(0)&=0,&
f_{n,P}(1)&=(\deg P_n)/2,&
\lim_{r\to+\infty}f_{n,P}(r)&=\deg P_n.
\end{align}
Moreover, all $f_{n,P}(r)$ are increasing functions in $\R^+$ since we have
\[
\frac{d}{dr}\frac{2r^{2k}+r^k}{1+r^k+r^{2k}}=\frac{k r^{k-1}(1+4r^k+r^{2k})}{(1+r^k+r^{2k})^2}>0.
\]
The existence and uniqueness of solution follows immediately.

It remains to prove the bounds on $r$. Since $f_{n,P}(r)$ is increasing,
it suffices to show that $f_{n,P}(r_0)<n$ and $f_{n,P}(1)\ge(\deg P_n)/2$.
The latter is true due to the second equation in \eqref{eq:fnP(1)}.
In order to see the former inequality, we argue as follows:
\begin{align*}
f_{n,P}(r_0)&<\sum_{k=1}^{n}\frac{k(2r_0^{2k}+r_0^k)}{1+r_0^k+r_0^{2k}}<\sum_{k=1}^{n}\alpha kr_0^k<\alpha\sum_{k=1}^{\infty} kr_0^k\\
&=\alpha\frac{r_0}{(1-r_0)^2}<\alpha(\log r_0)^{-2}=n.
\qedhere
\end{align*}
\end{proof}

%\begin{rem}
%Suppose that $e^{-\tau/n}$ is the saddle point for $z^{-m}A_{n,0}(z)$ for some $m$. Then we have an more exact asymptotic
%
%\[
%\frac{m}{n^2}=\frac{\pi^2}{9\tau^2}-\frac{\log(1+e^{-\tau}+e^{-2\tau})}{\tau}+\frac{1}{3}\Li_2(e^{-3\tau})-\Li_2(e^{-\tau})+O(1/n).
%\]
%\end{rem}

\section{The auxiliary polynomials $\tilde{P}_{n,j}(r)$}\label{seTilde}
As mentioned in Section~\ref{seOutline}, we will construct families of polynomials $\tilde{P}_{n,j}(r)$ satisfying \eqref{eqPTildeCond}. These polynomials are upper bounds for $|P_{n,j}(re^{i\theta})|$ with respect to $\theta$.
On the way, we also show that $D_{n,j}(q)$, $E_{n,j}(q)$ and $F_{n,j}(q)$
are polynomials in~$q$, as claimed in Remark~\ref{rem1}.
%In this part, we will give explicit forms for the upper-bound polynomials $\tilde{P}_{n,j}(r)$, which are required to satisfy \eqref{eqPTildeCond}.

To this end, we first note that the inequality $|f(re^{i\theta})|\leq f(r)$ trivially holds if $f$ is a polynomial with non-negative coefficients. Therefore, we proceed to factor out such parts from the polynomials $P_{n,j}(q)$, and bound the cofactor from above by the triangle inequality. Due to the relationship $F_{n,j}(q)=q^{3j}E_{n,j}(q)$, we will only explicitly write the factorization results for $P\in\{D,E\}$.

Using the definitions \eqref{eqExpansionD} and \eqref{eqExpansionE}, we arrive at the factorizations
\begin{align}
D_{n,j}(q)&=\left(\frac{q^{3j^2+3j}(q^3;q^3)_{n-3j}(q;q)_{3j+1}}{(q^3;q^3)_{3j+1}(q;q)_{n-3j}}\begin{bmatrix}3j+1 \\ j\end{bmatrix}_{q^3}\right)\left(\frac{(q^3;q^3)_{n-j-1}}{(q^3;q^3)_{n-3j-1}}\right), \label{eqDnjFactor1} \\
E_{n,j}(q)&=\left(\frac{q^{3j^2+3j}(q^3;q^3)_{n-3j}(q;q)_{3j+1}(1-q^3)}{(q^3;q^3)_{3j+1}(q;q)_{n-3j}(1-q^{9j+3})}\begin{bmatrix}3j+1\\ j\end{bmatrix}_{q^3}\right)\left(\frac{1+q^{3j+1}+q^{6j+2}}{1+q+q^2}\frac{(q^3;q^3)_{n-j-1}}{(q^3;q^3)_{n-3j-1}}\right). \label{eqEnjFactor1}
\end{align}
Here, $\left[\begin{smallmatrix}a \\ b\end{smallmatrix}\right]_{q}$ is the $q$-binomial coefficient, defined by $\left[\begin{smallmatrix}a \\ b\end{smallmatrix}\right]_{q}=\frac{(q;q)_{a}}{(q;q)_{a-b}(q;q)_{b}}$, which is known to be a polynomial in~$q$ with non-negative coefficients.

We claim that the first factors in \eqref{eqDnjFactor1} and \eqref{eqEnjFactor1} are polynomials with non-negative coefficients if $j\leq(n-1)/6$.
For $D_{n,j}(q)$, this is because the $q$-binomials and the polynomials \[
\frac{(q^3;q^3)_{b}(q;q)_{a}}{(q^3;q^3)_{a}(q;q)_{b}}=\prod_{k=a+1}^{b}(1+q^k+q^{2k})
\]
have non-negative coefficients. In the case of $E_{n,j}(q)$, the factor $\frac{1-q^3}{1-q^{9j+3}}\left[\begin{smallmatrix}3j+1 \\ j\end{smallmatrix}\right]_{q^3}$ is the $q$-analogue of the Fu\ss--Catalan numbers (see, for example, Stump \cite{MR2657714}), and it is also a polynomial with non-negative coefficients.

On the other hand, if $(n-1)/3>j>(n-1)/6$, then the first factors in \eqref{eqDnjFactor1} and \eqref{eqEnjFactor1} will no longer be polynomials. In these cases, we make the alternate factorizations
\begin{align}
D_{n,j}(q)&=\left(q^{3j^2}\begin{bmatrix}\floor{(n-1)/3}+j+1 \\ 2j+1\end{bmatrix}_{q^3}\right)\nonumber\\
&\qquad\times\left(\frac{(q^3;q^3)_{n-j-1}}{(q^3;q^3)_{\floor{(n-1)/3}+j+1}}\frac{(q;q)_{3j+1}(q^3;q^3)_{\floor{(n-1)/3}-j}}{(q^3;q^3)_j(q;q)_{n-3j-1}}\right),\label{eqDnjFactor2} \\
E_{n,j}(q)&=\left(q^{3j^2}\begin{bmatrix}\floor{(n-1)/3}+j+1 \\ 2j+1\end{bmatrix}_{q^3}\right)\nonumber\\
&\qquad\times\left((1-q)\frac{(q^3;q^3)_{n-j-1}}{(q^3;q^3)_{\floor{(n-1)/3}+j+1}}\frac{(q;q)_{3j}(q^3;q^3)_{\floor{(n-1)/3}-j}}{(q^3;q^3)_j(q;q)_{n-3j-1}}\right).\label{eqEnjFactor2}
\end{align}

In each of the equalities \eqref{eqDnjFactor1}--\eqref{eqEnjFactor2}, the first factor is a polynomial in~$q$ with non-negative coefficients, and the second factor is a product of factors of the form $1-q^k$ since
\[
\frac{(q;q)_{a}(q^3;q^3)_{\floor{b/3}}}{(q^3;q^3)_{\floor{a/3}}(q;q)_{b}}
=
\underset{3\nmid k}{\prod_{k=b+1}^{a}}(1-q^k),
\]
with the single exception of the factor $\frac {1+q^{3j+1}+q^{6j+2}}
{1+q+q^2}$ in \eqref{eqEnjFactor1}.
In particular, all of the second factors in
\eqref{eqDnjFactor1}--\eqref{eqEnjFactor2} are polynomials in~$q$,
with some negative coefficients though.

We note the trivial fact that $|1-q^k|\leq1+|q|^k$, as well as the slightly non-trivial fact that
\begin{align*}
\left|\frac{1+q^{3j+1}+q^{6j+2}}{1+q+q^2}\right|&\leq\frac{|1-q^{6j+3}|+|q-q^{3j+1}|+|q^{3j+2}-q^{6j+2}|}{|1-q^3|}\\
&\leq\frac{1-|q|^{6j+3}}{1-|q|^3}+\frac{|q|-|q|^{3j+1}}{1-|q|^3}+\frac{|q|^{3j+2}-|q|^{6j+2}}{1-|q|^3}\\
&\leq\frac{1+|q|+|q|^2-|q|^{6j+1}-|q|^{6j+2}-|q|^{6j+3}}{1-|q|^3}\\
&=\frac{1-|q|^{6j+1}}{1-|q|},
\end{align*}
as long as $|q|\le1$.

Based on these facts, we define the polynomials $\tilde{P}_{n,j}(r)$ to be the result of replacing $q$ by $r$ in the first parts of \eqref{eqDnjFactor1}--\eqref{eqEnjFactor2}, replacing every factor $1-q^k$ in the second parts of \eqref{eqDnjFactor1}--\eqref{eqEnjFactor2} by a corresponding factor $1+r^k$, and replacing $\frac{1+q^{3j+1}+q^{6j+2}}{1+q+q^2}$ in \eqref{eqEnjFactor1} by $\frac{1-r^{6j+1}}{1-r}$.

The immediate consequence of this definition are expressions for the quotients between successive $\tilde{P}_{n,j}(r)$'s. We have
\begin{align}
\frac{\tilde{D}_{n,j}(r)}{\tilde{D}_{n,j-1}(r)}&=\frac{r^{3j-3/2}(1+r^{3n-9j})(1+r^{3n-9j+3})(1+r^{3n-9j+6})}{(1+r^{n-3j+1}+r^{2n-6j+2})(1+r^{n-3j+2}+r^{2n-6j+4})(1+r^{n-3j+3}+r^{2n-6j+6})(1+r^{3n-3j})}\nonumber \\
&\kern1cm\times\frac{r^{3j+3/2}(1-r^{3j-1})(1-r^{3j+1})}{(1-r^{6j+3})(1-r^{6j})} \text{ for } 1\leq j\leq\floor{(n-1)/6},\label{eqDTildeRatio1}\\
\frac{\tilde{D}_{n,j}(r)}{\tilde{D}_{n,j-1}(r)}&=\frac{r^{3\floor{(n-1)/3}-3/2}(1+r^{3j-1})(1+r^{3j+1})(1+r^{n-3j})(1+r^{n-3j+1})(1+r^{n-3j+2})}{(1+r^{3\floor{(n-1)/3}+3j+3})(1+r^{3\floor{(n-1)/3}-3j+3})(1+r^{3n-3j})}\nonumber \\
&\kern1cm\times\frac{r^{6j-3\floor{(n-1)/3}+3/2}(1-r^{3\floor{(n-1)/3}+3j})(1-r^{3\floor{(n-1)/3}-3j})}{(1-r^{6j+3})(1-r^{6j})},\nonumber\\
&\kern6.5cm\text{ for } \floor{(n-1)/6}+1\leq j\leq\floor{(n-1)/3}, \label{eqDTildeRatio2}
\end{align}
for $\tilde{D}$, as well as
\begin{align}
\frac{\tilde{E}_{n,j}(r)}{\tilde{E}_{n,j-1}(r)}&=\frac{r^{3j-3/2}(1+r^{3n-9j})(1+r^{3n-9j+3})(1+r^{3n-9j+6})}{(1+r^{n-3j+1}+r^{2n-6j+2})(1+r^{n-3j+2}+r^{2n-6j+4})(1+r^{n-3j+3}+r^{2n-6j+6})(1+r^{3n-3j})}\nonumber \\
&\kern1cm\times\frac{r^{3j+3/2}(1-r^{3j-1})(1-r^{3j-2})(1-r^{6j+1})}{(1-r^{6j+3})(1-r^{6j})(1-r^{6j-5})} \text{ for } 1\leq j\leq\floor{(n-1)/6},\label{eqETildeRatio1}\\
\frac{\tilde{E}_{n,j}(r)}{\tilde{E}_{n,j-1}(r)}&=\frac{r^{3\floor{(n-1)/3}}(1+r^{3j-1})(1+r^{3j-2})(1+r^{n-3j})(1+r^{n-3j+1})(1+r^{n-3j+2})}{(1+r^{3\floor{(n-1)/3}+3j+3})(1+r^{3\floor{(n-1)/3}-3j+3})(1+r^{3n-3j})}\nonumber \\
&\kern1cm\times\frac{r^{6j-3\floor{(n-1)/3}}(1-r^{3\floor{(n-1)/3}+3j})(1-r^{3\floor{(n-1)/3}-3j})}{(1-r^{6j+3})(1-r^{6j})},\nonumber\\
&\kern6.5cm\text{ for } \floor{(n-1)/6}+1\leq j\leq\floor{(n-1)/3}. \label{eqETildeRatio2}
\end{align}
for $\tilde{E}$. Moreover, we trivially have
\begin{equation}\label{eqFtildeRatio}
\frac{\tilde{F}_{n,j}(r)}{\tilde{F}_{n,j-1}(r)}=r^{3}\frac{\tilde{E}_{n,j}(r)}{\tilde{E}_{n,j-1}(r)}.
\end{equation}
These relations will be used in the estimations of the tails and the remainders in Sections~\ref{seEps3} and \ref{seEps2}.

\section{The cut-off values}\label{seCutoff}
In order to get a good balance among the error terms $\epsilon_{i,P}$,
two cut-offs --- $\theta_0$ for the argument~$\theta$,
and $j_0$ for the summation index~$j$ --- will be chosen as
\begin{align}
\theta_0&=\frac13\frac{1-r}{1-r^n}, \label{eqCutoffTheta0}\\
j_0&=\lfloor \log_2n\rfloor, \label{eqCutoffJ0}
\end{align}
where $r$ is the value of the saddle point given by the unique solution
to \eqref{eqStationaryPoint}.

\begin{rem}\label{reQkRegion}
One consequence of the choice \eqref{eqCutoffTheta0} is that, whenever $q=re^{i\theta}$ with $0<r\leq1$ and $|\theta|<\theta_0$, we know that
\[
k|\theta|<\frac13\frac{k(1-r)}{(1-r^n)}\leq\frac13\frac{(-\log (r^k))}{(1-r^k)}
\]
for all $k$ with $1\leq k\leq n$. This means that the complex number $q^k$ belongs to the region
\begin{equation}\label{eqQkRegion}
\left\{Re^{i\Theta}\,\middle|\,|\Theta|<\frac13\frac{(-\log R)}{(1-R)}\right\}.
\end{equation}
\end{rem}

Having done all the preparatory work, we now dive into the estimations for the error terms $\epsilon_{i,P}$ in the next few sections.

\section{Bounding the primary peak error}\label{seEps0}
We begin this section by introducing a general bound on the relative errors for the approximation of a function by a Gau{\ss}ian.

\begin{lemma}\label{leGaussian}
Suppose that $x_0>0$ and $f\in C^3([-x_0,x_0];\C)$ satisfy
$f(x)=-gx^2/2+O(|x|^3)$ for some $g\in\R^+$. Let $h=\sup_{|x|\leq x_0}|f'''(x)|$.
Suppose further that $x_0<\frac{9g}{4h}$. Then we have
\[
\left|\sqrt{\frac{g}{2\pi}}\int_{-x_0}^{x_0}e^{f(x)}\,dx-1\right|\leq \erfc(x_0\sqrt{g/2})+1.1\times\frac{2\sqrt2}{3\sqrt\pi}\frac{h}{g^{3/2}}.
\]
\end{lemma}
\begin{proof}
Let $R_2(x)=f(x)+gx^2/2$. Taylor's theorem implies that
\[
|R_2(x)|\leq\frac{h}{6}|x|^3.
\]

We split the integral as follows:
\begin{align*}
\int_{-x_0}^{x_0}e^{f(x)}\,dx&=\int_{-x_0}^{x_0}e^{-gx^2/2}\,dx+\int_{-x_0}^{x_0}e^{-gx^2/2}\left(e^{R_2(x)}-1\right)\,dx\\
&=\frac{\sqrt{2\pi}}{\sqrt{g}}(1-\erfc(x_0\sqrt{g/2}))+\int_{-x_0}^{x_0}e^{-gx^2/2}\left(e^{R_2(x)}-1\right)\,dx.
\end{align*}
Therefore we have
\begin{align*}
\left|\sqrt{\frac{g}{2\pi}}\int_{-x_0}^{x_0}e^{f(x)}\,dx-1\right|&\leq \erfc(x_0\sqrt{g/2})+\sqrt{\frac{g}{2\pi}}\left|\int_{-x_0}^{x_0}e^{-gx^2/2}\left(e^{R_2(x)}-1\right)\,dx\right|\\
&<\erfc(x_0\sqrt{g/2})+\sqrt{\frac{2g}{\pi}}\int_{0}^{\frac{9g}{4h}}e^{-gx^2/2}\left(e^{hx^3/6}-1\right)\,dx.
\end{align*}
The last integral is then bounded using Lemma~\ref{leIneqBeta} by taking $u=g/2$ and $v=h/6$.
\end{proof}

\begin{lemma}\label{leIneqMainTerm}
Suppose that $r$ is chosen as the saddle point as described in
Lemma~\ref{leRadiusBound}. Then, for all $n\geq1500>120(9+2\sqrt3)$, we have
\[
\epsilon_{0,P}(n,m,r)<\frac{7\sqrt{2}}{\sqrt{3\pi\lambda}}+\erfc\sqrt{\frac{\lambda}{84}},
\]
where $\lambda=\frac{r-r^{n+1}}{1-r}$.
\end{lemma}

\begin{proof}
Note that the choice of $r$ as the saddle point of $P_{n,0}(re^{i\theta})e^{-im\theta}$ ensures that the Taylor expansion of $\log P_{n,0}(re^{i\theta})e^{-im\theta}$ at $\theta=0$ has a vanishing linear term. Thus we can use Lemma~\ref{leGaussian} to bound the relative error $\epsilon_{0,P}(n,m,r)$. We define
\[
h_{3,P}(n,r)=\sup_{|\theta|\le\theta_0}\left|\frac{\partial^3}{\partial\theta^3}\log P_{n,0}(re^{i\theta})\right|,
\]

Lemma~\ref{leGaussian} immediately allows us to conclude
\begin{equation}\label{1}
\epsilon_{0,P}(n,m,r)\leq\erfc(\theta_0\sqrt{g_P(n,r)/2})+1.1\times\frac{2\sqrt2}{3\sqrt\pi}\frac{h_P(n,r)}{g_P(n,r)^{3/2}},
\end{equation}
provided that $\theta_0<\frac{9g_P(n,r)}{4h_{3,P}(n,r)}$.

The subsequent arguments in this part exploit some inequalities for the quantities $g_P(n,r)$, $h_{3,P}(n,r)$ and $\theta_0$ to verify the conditions of Lemma~\ref{leGaussian}.

We start by establishing simpler bounds on them. For the sake of simplicity, we write $g$ and $h$ for $g_P(n,r)$ and $h_{3,P}(n,r)$ in the subsequent arguments.

The definition of $h$ implies that
\begin{align*}
h=h_{3,P}(n,r)\leq\sum_{k=1}^{n-1}\sup_{|\theta|\le\theta_0}\left|\frac{ik^3q^k(1-q^{2k})(1+7q^k+q^{2k})}{(1+q^k+q^{2k})^3}\right|,
\end{align*}
where $q=re^{i\theta}$.

Therefore, an upper bound for $h$ can be directly inferred from \eqref{eqIneqH}:
\begin{equation}\label{eqHUpperBound}
h\leq\frac75\sum_{k=1}^{n}k^3r^k.
\end{equation}

On the other hand, \eqref{eqGPFormula} and the elementary inequality $\frac65>\frac{1+4r+r^2}{(1+r+r^2)^2}\geq\frac23$ lead to the following bounds for $g$:
\begin{equation}\label{eqGUpperBound}
g<\frac65\sum_{k=1}^{n}k^2r^k,
\end{equation}
as well as
\begin{align}
\nonumber
g&\geq\frac23\sum_{k=1}^{n}k^2r^k-r-n^2r^n\\
\nonumber
&\geq \left(\frac23-\frac{n^2}{\sum_{k=1}^{n}k^2}-\left(\frac{1-r}{1-r^{n/2}}\right)^2\right)\sum_{k=1}^{n}k^2r^k\\
\nonumber
&>\left(\frac23-\frac{3}{n}-\frac{\alpha}{n}\right)\sum_{k=1}^{n}k^2r^k\\
&>\left(\frac23-\frac{1}{360}\right)\sum_{k=1}^{n}k^2r^k,
\label{eqGLowerBound}
\end{align}
where we use the inequality $\sum_{k=1}^{n}k^2r^k\geq\sum_{k=1}^{n}k^2r^n$, as well as
\begin{align*}
\sum_{k=1}^{n}k^2r^k&\geq \sum_{k=1}^{n}kr^k=\frac{r}{(1-r)^2}\left(1+r^n-2r^n\left(1+\frac{n}{2}(1-r)\right)\right)\\
&\geq\frac{r}{(1-r)^2}\left(1+r^n-2r^n\left(1+\frac{n}{2}(r^{-1}-1)\right)\right)\\
&\geq\frac{r}{(1-r)^2}\left(1+r^n-2r^nr^{-n/2}\right)=r\left(\frac{1-r^{n/2}}{1-r}\right)^2,
\end{align*}
and
\[
n\left(\frac{1-r}{1-r^{n/2}}\right)^2<n\left(\frac{1-\exp(-\sqrt{\alpha/n})}{1-\exp(-\sqrt{\alpha n}/2)}\right)^2<\alpha,
\]
which is a consequence of Lemma~\ref{leRadiusBound}.

Having established the bounds above, we can establish some relationships among $g$, $h$, $\theta_0$ and $\lambda=\sum_{k=1}^nr^k=\frac{r-r^{n+1}}{1-r}$.

The inequalities \eqref{eqHUpperBound}, \eqref{eqGLowerBound} and \eqref{eqIneqR02_3} imply that
\begin{align*}
\theta_0&=\frac{r}{3\lambda}\leq\frac{r}{3}\frac{(r^2+4r+1)}{r(r+1)}\frac{\sum_{k=1}^nk^2r^k}{\sum_{k=1}^nk^3r^k}\\
&\leq\frac{\sum_{k=1}^nk^2r^k}{\sum_{k=1}^nk^3r^k}<\frac{g/\left(\frac23-\frac{1}{360}\right)}{5h/7}\\
&\leq\frac{45g/28}{5h/7}=\frac{9g}{4h}.
\end{align*}

We also infer from \eqref{eqHUpperBound}, \eqref{eqGLowerBound} and \eqref{eqIneqR222_033} that
\begin{align*}
\frac{1.1\times2\sqrt2}{3\sqrt{\pi}}\frac{h}{g^{3/2}}
&<\frac{2\sqrt2}{3}\frac{\frac75\times1.1}{\sqrt{\pi}\left(\frac23-\frac4{5n}\right)^{3/2}}\frac{\sum_{k=1}^nk^3r^k}{\left(\sum_{k=1}^nk^2r^k\right)^{3/2}}\\
&\leq\frac{2\sqrt2}{3}\frac{\frac75\times1.1}{\sqrt{\pi}\left(\frac23-\frac{1}{360}\right)^{3/2}}\sqrt\frac{(1+4r+r^2)^2}{(1+r)^3\sum_{k=1}^nr^k}<\frac{2\sqrt2}{3}\frac{\frac75\times\frac{10}{9}}{\sqrt{\pi}(\frac23)^{3/2}}\sqrt\frac{9}{2\lambda}\\
&=\frac{7\sqrt{2}}{\sqrt{3\pi\lambda}},
\end{align*}
where we used the numerical inequality $1.1\left(\frac23-\frac{1}{360}\right)^{-3/2}<\frac{10}{9}\left(\frac23\right)^{-3/2}$.

Finally, to bound the complementary error function in \eqref{1}, which is equivalent to bound $g\theta_0^2$ from below, we invoke \eqref{eqGLowerBound} and \eqref{eqIneqR2_000} to see that
\[
g\theta_0^2>\left(\frac23-\frac{1}{360}\right)\frac{r}{3}\left(\frac{1-r^n}{1-r}\right)^3\left(\frac{1}{3}\frac{1-r}{1-r^n}\right)^2>\frac{\lambda}{42},
\]
and therefore
\[
\erfc(\theta_0\sqrt{g/2})>\erfc(\sqrt{\lambda/84}).
\qedhere
\]
\end{proof}

\section{Bounding the secondary peaks}\label{seEps1}
The error terms $\epsilon_{1,P}(n,r)$ related to the secondary peaks concern the quotients $\left|\frac{P_{n,j}(re^{i\theta})}{P_{n,0}(re^{i\theta})}\right|$. To bound these quotients from above, we look at the quotients of two consecutive polynomials.

\begin{align}
\frac{D_{n,j}(q)}{D_{n,j-1}(q)}&=\frac{q^{6j}(1-q^{n-3j})(1-q^{n-3j+1})(1-q^{n-3j+2})(1-q^{3j-1})(1-q^{3j+1})}{(1-q^{3n-3j})(1-q^{6j+3})(1-q^{6j})},\label{eqQuotientD}\\
\frac{E_{n,j}(q)}{E_{n,j-1}(q)}&=\frac{q^{6j}(1-q^{n-3j})(1-q^{n-3j+1})(1-q^{n-3j+2})(1-q^{3j-1})(1-q^{3j-2})}{(1-q^{3n-3j})(1-q^{6j+3})(1-q^{6j})}.\label{eqQuotientE}
\end{align}
In order to bound these quotients from above, we introduce a general result that can be used to control the rational function $\frac{1-q^a}{1-q^b}$ in certain domains of the complex plane.

\begin{lemma}\label{leIneqSinh}
Suppose $0<a\leq b$, and\/ $0<c\leq\pi/b$. Then for all $z\in\C$ such that $(\Im z)^2\leq(\Re z)^2+c^2$, we have
\[
\left|\frac{\sinh az}{\sinh bz}\right|\leq\frac{\sin ac}{\sin bc}.
\]
\end{lemma}

\begin{proof}
We make use of the infinite products
$$
\sinh z=z\prod_{k=1}^{\infty}\left(1+\frac{z^2}{k^2\pi^2}\right)\\
$$
and
$$
\sin z=z\prod_{k=1}^{\infty}\left(1-\frac{z^2}{k^2\pi^2}\right).
$$

We claim that under the assumptions of this lemma, we have
$$
\left|\frac{k^2\pi^2+a^2z^2}{k^2\pi^2+b^2z^2}\right|\leq\frac{k^2\pi^2-a^2c^2}{k^2\pi^2-b^2c^2},
$$
from which the lemma follows after taking the product over all $k\ge1$.

In order to prove this inequality, we write $z^2=x+iy$ and $u=k\pi$, so that $x\geq-c^2$ and $ac\leq bc\leq u$. Now the absolute value can be written as
\[
\left|\frac{u^2+a^2z^2}{u^2+b^2z^2}\right|=\sqrt\frac{(u^2+a^2x)^2+a^4y^2}{(u^2+b^2x)^2+b^4y^2},
\]
and the inequality can be proved by the manipulation
\begin{align*}
\left((u^2+b^2x)^2+b^4y^2\right)(u^2-a^2c^2)^2&-{}\left((u^2+a^2x)^2+a^4y^2\right)(u^2-b^2c^2)^2\\
&=u^2(b^2-a^2)\left[(u^2-a^2c^2)((x+c^2)(u^2+b^2x)+b^2y^2)\right.\\
&\kern1cm\left.+(u^2-b^2c^2)((x+c^2)(u^2+a^2x)+a^2y^2)\right]\\
&\geq0.
\qedhere
\end{align*}
\end{proof}

\begin{lemma}\label{leIneqPQuotient}
Suppose that $r_0$ and $\theta_0$ are as defined as in \eqref{eqCutoffR0} and \eqref{eqCutoffTheta0}, respectively. Then, for all $j\in[1,\lfloor n/3\rfloor)$, all $q=re^{i\theta}\in\C$ such that $r\in(r_0,1]$, and $|\theta|<\theta_0$, we have
$$
\left|\frac{D_{n,j}(q)}{D_{n,j-1}(q)}\right|<(1.005+1.3/n)\frac{(3j+1)(3j-1)}{18j(2j+1)}\left(\frac{(j+1)\pi/n}{\sin (j+1)\pi/n}\right)^2,
$$
and
$$
\left|\frac{E_{n,j}(q)}{E_{n,j-1}(q)}\right|<|q|^{-3/2}(1.005+1.3/n)\frac{(3j-1)(3j-2)}{18j(2j+1)}\left(\frac{(j+1)\pi/n}{\sin (j+1)\pi/n}\right)^2.
$$
\end{lemma}

\begin{proof}
We write $z=\frac12\log q$ so that $e^{2z}=q$ and $(q^a-1)=q^{a/2}\sinh az$. Note that the conditions on $q$ imply the inequality
\begin{equation}\label{eqZIneq1}
|\Im z|\leq \frac16\frac{(1-e^{2\Re z})}{(1-e^{2n\Re z})}.
\end{equation}

We claim that the inequality
\begin{equation}\label{eqZIneq2}
\frac16\frac{(1-e^{-2u})}{(1-e^{-2nu})}<\max\left(u,\frac{1}{3n}\right)<\sqrt{u^2+\frac{1}{9n^2}}
\end{equation}
holds for all $n\geq1$ and all $u\geq0$. This can be proved by observing that
\[
\frac16\frac{(1-e^{-2u})}{(1-e^{-2nu})}<\frac{u}{3}\frac{1}{(1-e^{-2nu})}<\frac{u}{3}\frac{1}{(1-e^{-1})}<u
\]
if $u>\frac{1}{2n}$, and
\[
\frac16\frac{(1-e^{-2u})}{(1-e^{-2nu})}\leq\frac16\frac{(1-e^{-1/n})}{(1-e^{-1})}<\frac{1}{6n}\frac{1}{(1-e^{-1})}<\frac{1}{3n}
\]
if $u\leq\frac{1}{2n}$.
Therefore, \eqref{eqZIneq1} and \eqref{eqZIneq2} imply that $z$ satisfies the condition in Lemma~\ref{leIneqSinh} with $c=\frac{1}{3n}<\frac{\pi}{6n}$. Lemma~\ref{leIneqSinh} now says that, for any $a,b\in\R^+$ where $0<a\leq b\leq 6n$, we have
\begin{equation}\label{eqIneqSinhAlternateForm}
\left|\frac{q^{(b-a)/2}(1-q^a)}{1-q^b}\right|\leq\frac{\sin ac}{\sin bc}.
\end{equation}

We use \eqref{eqIneqSinhAlternateForm} to bound various parts on the right-hand sides of \eqref{eqQuotientD} and \eqref{eqQuotientE}. We have
\begin{align*}
\left|\frac{q^{3j-3}(1-q^{3n-9j+6})}{1-q^{3n-3j}}\right|&\leq\frac{\sin \frac{n-3j+2}{2n}\pi}{\sin \frac{n-j}{2n}\pi}\leq1,\\
\left|\frac{q^{3j-3/2}(1-q^{3n-9j+3})}{1-q^{3n-3j}}\right|&\leq\frac{\sin \frac{n-3j+1}{2n}\pi}{\sin \frac{n-j}{2n}\pi}\leq1,\\
\end{align*}
as well as
\begin{align*}
\left|\frac{q^{(c+d-a-b)/2}(1-q^{a})(1-q^{b})}{(1-q^{c})(1-q^{d})}\right|&\leq\frac{\sin \frac{a}{6n}\pi\,\sin \frac{b}{6n}\pi}{\sin \frac{c}{6n}\pi\,\sin \frac{d}{6n}\pi}\\
&<\frac{ab}{cd}\left(\frac{c\pi/6n}{\sin c\pi/6n}\frac{d\pi/6n}{\sin d\pi/6n}\right)<\frac{ab}{cd}\left(\frac{(j+1)\pi/n}{\sin (j+1)\pi/n}\right)^2,
\end{align*}
for $(a,b,c,d)=(3j-1,3j+1,6j+3,6j)$ or $(3j-1,3j-2,6j+3,6j)$.

It remains to bound the factor
\[
\left|\frac{(1-q^{k-1})(1-q^{k+1})}{1+q^k+q^{2k}}\right|,
\]
where $k=n-3j+1$ or $n-3j+2$. Here we make use of \eqref{eqIneqJ1} and \eqref{eqIneqJ2} (recall that $q^k$ belongs to the region \eqref{eqQkRegion}) to conclude that
\begin{align*}
\left|\frac{(1-q^{k-1})(1-q^{k+1})}{1+q^k+q^{2k}}\right|
&\leq\left|\frac{(1-q^k)^2}{1+q^k+q^{2k}}\right|+\left|\frac{q^{k-1}(1-q)^2}{1+q^k+q^{2k}}\right|\\
&<1.005+1.002|1-q|^2\\
&<1.005+1.002((1-r)^2+\theta_0^2))\\
&\leq1.005+1.002\left(\frac{2}{\sqrt3n}(1+\frac{1}{9})\right)\\
&<1.005+1.3/n.
\qedhere
\end{align*}
\end{proof}

These bounds allow us to obtain upper bounds for the first factor in the expression \eqref{eq:ep1} of the error term $\epsilon_{1,P}(n,r)$.
\begin{lemma}\label{leIneqSmallJPeak1}
Suppose $n>7000$, and that $r_0$, $j_0$ and $\theta_0$ are as defined as in \eqref{eqCutoffR0}, \eqref{eqCutoffJ0} and \eqref{eqCutoffTheta0}, respectively. Then, for all $r\in(r_0,1]$, we have
\begin{align*}
\sum_{j=1}^{j_0}\sup_{|\theta|<\theta_0}\left|\frac{D_{n,j}\left(re^{i\theta}\right)}{D_{n,0}\left(re^{i\theta}\right)}\right|&<0.187,\\
\sum_{j=1}^{j_0}\sup_{|\theta|<\theta_0}\left|\frac{E_{n,j}\left(re^{i\theta}\right)}{E_{n,0}\left(re^{i\theta}\right)}\right|&<0.043,\\
\sum_{j=1}^{j_0}\sup_{|\theta|<\theta_0}\left|\frac{F_{n,j}\left(re^{i\theta}\right)}{F_{n,0}\left(re^{i\theta}\right)}\right|&<0.043.
\end{align*}
\end{lemma}

\begin{proof}
We first make use of Lemma~\ref{leIneqPQuotient}, and we notice that the condition $n>7000$ and the choice of $j_0$ imply that $(j+1)/n<\frac{\log n+\log 2}{n\log 2}<\frac{1}{500}$. Therefore, the terms involving $n$ in Lemma~\ref{leIneqPQuotient} can be bounded above by
\[
(1.005+1.3/7000)\left(\frac{\pi/500}{\sin \pi/500}\right)^2<1.006.
\]

This implies
$$
\left|\frac{D_{n,j}(q)}{D_{n,0}(q)}\right|<1.006^j\prod_{k=1}^{j}\frac{(3k+1)(3k-1)}{18k(2k+1)}=\frac{1.006^j}{27^j}\binom{3j+1}{j},
$$
and
$$
\left|\frac{E_{n,j}(q)}{E_{n,0}(q)}\right|<1.006^j\prod_{k=1}^{j}|q|^{-3k/2}\frac{(3k-1)(3k-2)}{18k(2k+1)}=\frac{(1.006|q|^{-3/2})^j}{27^j(3j+1)}\binom{3j+1}{j},
$$
for all $j$ with $1\leq j\leq j_0$. The relationship \eqref{eqExpansionF} implies a similar inequality for $F_{n,j}$, namely
\begin{align*}
\left|\frac{F_{n,j}(q)}{F_{n,0}(q)}\right|&<1.006^j\prod_{k=1}^{j}|q|^{3k/2}\frac{(3k-1)(3k-2)}{18k(2k+1)}=\frac{(1.006|q|^{3/2})^j}{27^j(3j+1)}\binom{3j+1}{j}.
\end{align*}

The bounds stated in the lemma can be obtained by noticing that
$$|q|^{-3/2}\leq r_0^{-3/2}=\exp\left(\sqrt{\sqrt{3}/n}\right)<\exp\left(\sqrt{\sqrt{3}/7000}\right)<1.0003,$$
and by estimating
\begin{align*}
\sum_{j=1}^{\infty}\frac{1.006^j}{27^j}\binom{3j+1}{j}&\approx0.18618<0.187,\\
\sum_{j=1}^{\infty}\frac{(1.006\times1.0003)^j}{27^j(3j+1)}\binom{3j+1}{j}&\approx0.04219<0.043,\\
\sum_{j=1}^{\infty}\frac{1.006^j}{27^j(3j+1)}\binom{3j+1}{j}&\approx0.04218<0.043.
\qedhere
\end{align*}
\end{proof}

It remains to deal with the second factor in \eqref{eq:ep1}, namely
\[
\frac{\sqrt{g_P(n,r)}}{\sqrt{2\pi}}\int_{-\theta_0}^{\theta_0}\left|\frac{P_{n,0}(re^{i\theta})}{P_{n,0}(r)}\right|\,d\theta.
\]
The argument below is parallel to the one in Section~\ref{seEps0}. The main difference is that the integrand $\left|\frac{P_{n,0}(re^{i\theta})}{P_{n,0}(r)}\right|$ is an even function in~$\theta$, so the error term in the Taylor expansion is of order four instead of order three. We first present an analogue of Lemma~\ref{leGaussian}.
\begin{lemma}\label{leGaussian2}
Suppose that $x_0>0$ and that $f\in C^4([-x_0,x_0])$ is an even function that satisfies
$f(x)=-gx^2/2+O(|x|^4)$ for some $g\in\R^+$. Let $h=\sup_{|x|\leq x_0}|f^{(4)}(x)|$.
Suppose further that $x_0^2<\frac{27g}{8h}$. Then we have
\[
\sqrt{\frac{g}{2\pi}}\int_{-x_0}^{x_0}e^{f(x)}\,dx\leq 1+\frac{\sqrt2}{9\sqrt\pi}\frac{h^{1/2}}{g}.
\]
\end{lemma}
\begin{proof}
Let $R(x)=f(x)+gx^2/2$. Taylor's theorem implies that
\[
|R(x)|\leq\frac{h}{24}|x|^4.
\]

Similar to the proof of Lemma~\ref{leGaussian}, we argue that
\begin{align*}
\sqrt{\frac{g}{2\pi}}\int_{-x_0}^{x_0}e^{f(x)}\,dx&=1-\erfc(x_0\sqrt{g/2})+\sqrt{\frac{g}{2\pi}}\int_{-x_0}^{x_0}e^{-gx^2/2}\left(e^{R(x)}-1\right)\,dx\\
&\leq 1+\sqrt{\frac{g}{2\pi}}\int_{-x_0}^{x_0}e^{-gx^2/2}\left(e^{h|x|/24}-1\right)\,dx\\
&=1+\sqrt{\frac{2g}{\pi}}\int_{0}^{\sqrt{\frac{27g}{8h}}}e^{-gx^2/2}\left(e^{h|x|/24}-1\right)\,dx.
\end{align*}
The last integral is then bounded using Lemma~\ref{leIneqBeta2} by taking $u=g/2$ and $v=h/24$.
\end{proof}

\begin{lemma}\label{leIneqSmallJPeak2}
Suppose that $n\geq1500>120(9+2\sqrt3)$, and $\theta_0$ is defined as in \eqref{eqCutoffTheta0}. Then we have
\[
\frac{\sqrt{g_P(n,r)}}{\sqrt{2\pi}}\int_{-\theta_0}^{\theta_0}\left|\frac{P_{n,0}(re^{i\theta})}{P_{n,0}(r)}\right|\,d\theta\leq1+\frac{\sqrt{5}}{3\sqrt{3\lambda}},
\]
where $\lambda=\frac{r-r^{n+1}}{1-r}$.
\end{lemma}
\begin{proof}
Note that the integrand is an even function in~$\theta$, so we can use Lemma~\ref{leGaussian2} to bound the integral. We define
\[
h_{4,P}(n,r)=\sup_{|\theta|\le\theta_0}\left|\frac{\partial^4}{\partial\theta^4}\log P_{n,0}(re^{i\theta})\right|,
\]
Lemma~\ref{leGaussian} immediately allows us to conclude
\begin{equation}\label{2}
\frac{\sqrt{g_P(n,r)}}{\sqrt{2\pi}}\int_{-\theta_0}^{\theta_0}\left|\frac{P_{n,0}(re^{i\theta})}{P_{n,0}(r)}\right|\,d\theta\leq1+\frac{\sqrt2}{9\sqrt\pi}\frac{h_{4,P}(n,r)^{1/2}}{g_P(n,r)},
\end{equation}
provided that the condition $\theta_0^2<\frac{27g_P(n,r)}{8h_{4,P}(n,r)}$
is satisfied.

The subsequent arguments in this part exploit some inequalities for the quantities $g_P(n,r)$, $h_{4,P}(n,r)$ and $\theta_0$ to verify the conditions of Lemma~\ref{leGaussian2}.

We start by establishing simpler bounds on them. For the sake of simplicity, we write $g$ and $h$ for $g_P(n,r)$ and $h_{4,P}(n,r)$ in the subsequent arguments.

The definition of $h$ implies that
\[
h=h_{4,P}(n,r)\leq\sum_{k=1}^{n-1}\sup_{|\theta|\le\theta_0}\left|\frac{k^4q^k(1+12q^k-12q^{2k}-56q^{3k}-12q^{4k}+12q^{5k}+q^{6k})}{(1+q^k+q^{2k})^4}\right|,
\]
where $q=re^{i\theta}$.

Therefore, an upper bound for $h$ can be directly inferred from \eqref{eqIneqH2}:
\begin{equation}\label{eqHUpperBound2}
h<\frac53\sum_{k=1}^{n}k^4r^k.
\end{equation}

On the other hand, we recall the upper and lower bounds on $g$ from \eqref{eqGUpperBound} and \eqref{eqGLowerBound}. We establish some relationships among $g$, $h$, $\theta_0$ and $\lambda=\sum_{k=1}^nr^k=\frac{r-r^{n+1}}{1-r}$.

The inequalities \eqref{eqHUpperBound2}, \eqref{eqGLowerBound} and \eqref{eqIneqR002_4} imply that
\begin{align*}
\theta_0^2=\frac{r^2}{9\lambda^2}&\leq\frac{r^2}{9}\frac{1+10r+r^2}{r^2}\frac{\sum_{k=1}^nk^2r^k}{\sum_{k=1}^nk^4r^k}\\
&\leq\frac43\frac{\sum_{k=1}^nk^2r^k}{\sum_{k=1}^nk^3r^k}<\frac43\frac{g/\left(\frac23-\frac{1}{360}\right)}{3h/5}\\
&\leq\frac{27g}{8h}.
\end{align*}
Moreover,
from \eqref{eqHUpperBound2}, \eqref{eqGLowerBound} and \eqref{eqIneqR22_04},
we also infer that
\begin{align*}
\frac{\sqrt2}{9\sqrt\pi}\frac{h^{1/2}}{g}&<\frac{\sqrt2}{9\sqrt\pi}\frac{\sqrt{5/3}}{\frac23-\frac{1}{360}}\frac{\left(\sum_{k=1}^nk^4r^k\right)^{1/2}}{\sum_{k=1}^nk^2r^k}\\
&\leq\frac{\sqrt2}{9\sqrt3}\frac{\sqrt{5/3}}{2/3}\sqrt\frac{1+10r+r^2}{(1+r)\left(\sum_{k=1}^nr^k\right)}\leq\frac{\sqrt5}{9\sqrt2}\sqrt\frac{6}{\lambda}\\
&=\frac{\sqrt{5}}{3\sqrt{3\lambda}}.
\qedhere
\end{align*}
\end{proof}
By combining Lemmas~\ref{leIneqSmallJPeak1} and \ref{leIneqSmallJPeak2}, we arrive at our bound for the error term $\epsilon_{1,P}(n,r)$.
\begin{lemma}\label{leIneqSmallJPeak}
For all $n>7000$ and all $r$ with $0<r\leq1$, we have
\begin{align*}
\epsilon_{1,D}(n,r)&<0.187\left(1+\frac{\sqrt{5}}{3\sqrt{3\lambda}}\right), \\
 \epsilon_{1,E}(n,r)&<0.043\left(1+\frac{\sqrt{5}}{3\sqrt{3\lambda}}\right),\\
\epsilon_{1,F}(n,r)&<0.043\left(1+\frac{\sqrt{5}}{3\sqrt{3\lambda}}\right).
\end{align*}
\end{lemma}

\section{Bounding the remainders}\label{seEps3}
The reason we estimate the remainder parts before the tail is that certain results in this section, namely upper bounds for the ratios $\left|\frac{\tilde{P}_{n,j}(r)}{\tilde{P}_{n,j-1}(r)}\right|$, will also be used in bounding the tails from above.

\begin{lemma}\label{leIneqPTildeQuotient} Suppose that $n\in\Z^+$, and $0<r\leq1$.
For all $j\in[1,\floor{(n-1)/6}]$, we have
\begin{align*}
\left|\frac{\tilde{D}_{n,j}(r)}{\tilde{D}_{n,j-1}(r)}\right|&<\frac{(3j-1)(3j+1)}{18j(2j+1)}, & \left|\frac{\tilde{E}_{n,j}(r)}{\tilde{E}_{n,j-1}(r)}\right|&<\frac{r^{-3/2}(3j-1)(3j-2)(6j+1)}{18j(2j+1)(6j-5)}.
\end{align*}
On the other hand, for all $j\in[\floor{(n-1)/6}+1,\floor{(n-1)/3}]$, we have
\begin{align*}
\left|\frac{\tilde{D}_{n,j}(r)}{\tilde{D}_{n,j-1}(r)}\right|&<\frac{4(\floor{(n-1)/3}-j)}{3j-\floor{(n-1)/3}+1}, & \left|\frac{\tilde{E}_{n,j}(r)}{\tilde{E}_{n,j-1}(r)}\right|&<\frac{4r^{-3/2}(\floor{(n-1)/3}-j)}{(3j-\floor{(n-1)/3}+1)}.
\end{align*}
\end{lemma}
\begin{proof}
We claim that the first factors in \eqref{eqDTildeRatio1} and \eqref{eqETildeRatio1} do not exceed~1, and the first factors in \eqref{eqDTildeRatio2} and \eqref{eqETildeRatio2} do not exceed 4. The claims about \eqref{eqDTildeRatio1} and \eqref{eqETildeRatio1} are proved by using the inequality $1+r^{3k}<1+r^{k+1}+r^{2k+2}$ for $k=n-3j, n-3j+1,n-3j+2$, and the claims about \eqref{eqDTildeRatio2} and \eqref{eqETildeRatio2} are proved by observing that
\[
\frac{(1+r^{n-3j})(1+r^{n-3j+1})(1+r^{n-3j+2})}{(1+r^{3\floor{(n-1)/3}-3j+3})}\leq4,
\]
as well as the inequality $r^{(b-a)/2}\frac{1+r^a}{1+r^b}\leq1$, valid for $0<r<1$ and $0<a<b$.

The second factors in \eqref{eqDTildeRatio1} and \eqref{eqETildeRatio1} can be estimated using the inequality $r^{(b-a)/2}\frac{1-r^{a}}{1-r^{b}}\leq\frac{a}{b}$, valid for all $r\in\R$ and $b\geq a>0$. (This can be considered as a limiting form of Lemma~\ref{leIneqSinh} when $c\to0$.) In order to deal with the factor $\frac{1-r^{6j+1}}{1-r^{6j-5}}$, we use the fact that the function $a\mapsto \frac{1-r^a}{a}$ is decreasing in $a$ if $0<r\leq1$. This concludes the proof of the first part of the lemma.

We cannot directly use the same method for the second factors in \eqref{eqDTildeRatio2} and \eqref{eqETildeRatio2} because, in each case, one exponent in the numerator, namely $3\floor{(n-1)/3}+3j$, would be larger than both exponents in the denominator. Instead, we argue that if $a\geq c\geq d\geq b$ and $c+d\geq a+b$, then we have
\begin{align*}
\frac{r^{(c+d-a-b)/2}(1-r^{a})(1-r^{b})}{(1-r^{c})(1-r^{d})}
&\leq\frac{b}{c+d-a}\frac{(1-r^{a})(1-r^{c+d-a})}{(1-r^{c})(1-r^{d})}\\
&\leq\frac{b}{c+d-a}.
\end{align*}

Insertion of specific values of $a,b,c,d$ from \eqref{eqDTildeRatio2} and \eqref{eqETildeRatio2} into the above inequality concludes the proof.
\end{proof}
%\begin{equation}
%\begin{split}
%&\phantom{=}\frac{r^{6j-3-3\floor{n/3}}(1-r^{3\floor{n/3}+3j})(1-r^{3\floor{n/3}-3j+3})}{(1-r^{6j-3})(1-r^{6j})}\\
%&=\frac{(1-r^{3\floor{n/3}+3j})(r^{6j-3-3\floor{n/3}}-r^{3j})}{(1-r^{6j-3})(1-r^{6j})}\\
%&\leq\frac{\floor{n/3}-j+1}{3j-\floor{n/3}-1}\frac{(1-r^{3\floor{n/3}+3j})(1-r^{9j-3-3\floor{n/3}})}{(1-r^{6j-3})(1-r^{6j})}\\
%&\leq\frac{\floor{n/3}-j+1}{3j-\floor{n/3}-1}.
%\end{split}
%\end{equation}

\begin{lemma}\label{leIneqLargeJ}
Suppose that $n>7000$, and that $r_0$, $j_0$ and $\theta_0$ are as defined as in \eqref{eqCutoffR0}, \eqref{eqCutoffJ0} and \eqref{eqCutoffTheta0}, respectively. Then, for all $r\in(r_0,1]$, we have
\begin{align*}
\epsilon_{3,D}(n,r)&<0.004, & \epsilon_{3,E}(n,r)&<0.008, &\epsilon_{3,F}(n,r)&<0.008.
\end{align*}
\end{lemma}

\begin{proof}
Lemma~\ref{leIneqPTildeQuotient} implies the following inequalities for $\tilde{D}_{n,j}$ and $\tilde{E}_{n,j}$:
\begin{align}
\frac{\tilde{D}_{n,j}(r)}{D_{n,0}(r)}&\leq\binom{3j+1}{j}3^{-3j},
& \text{ for } 0\leq j\leq\floor{(n-1)/6},\label{eqDTildeUpperBound1} \\
\frac{\tilde{D}_{n,j}(r)}{D_{n,\floor{(n-1)/6}}(r)}&\leq\prod_{k=\floor{(n-1)/6}+1}^{j}\frac{4(\floor{(n-1)/3}-k+1)}{(3k-\floor{(n-1)/3}-1)},
& \text{ for } \floor{(n-1)/6}<j\leq \floor{(n-1)/3},\label{eqDTildeUpperBound2} \\
\frac{\tilde{E}_{n,j}(r)}{E_{n,0}(r)}&\leq\frac{6j+1}{3j+1}r^{-3j/2}\binom{3j+1}{j}3^{-3j},
& \text{ for } 0\leq j\leq\floor{(n-1)/6},\label{eqETildeUpperBound1} \\
\frac{\tilde{E}_{n,j}(r)}{E_{n,\floor{(n-1)/6}}(r)}&\leq \prod_{k=\floor{(n-1)/6}+1}^{j}\frac{4(\floor{(n-1)/3}-k+1)}{(3k-\floor{(n-1)/3}-1)}, &\text{ for } \floor{(n-1)/6}<j\leq \floor{(n-1)/3}.\label{eqETildeUpperBound2}
\end{align}

We observe that the factor $\frac{4(K-k+1)}{(3k-K-1)}$ does not exceed $4$ for all $k\in[(K+1)/2,K]$, and it does not exceed $1$ for all $k\in[5(K+1)/7,K]$.
Therefore, the right-hand sides of \eqref{eqDTildeUpperBound2} and \eqref{eqETildeUpperBound2} (where $K=\floor{(n-1)/3}$) can be bounded above by
\[
4^{(\frac57-\frac12)(K+1)}\leq4^{\frac{3}{14}(\floor{n/3}+1)}<2^{n/7+1}.
\]

Taking into account the inequality $\binom{3j+1}{j}3^{-3j}<4^{-j}\sqrt{\frac{27}{16\pi j}}$, we calculate
\begin{align*}
\sum_{j>j_0}\frac{\tilde{D}_{n,j}(r)}{D_{n,0}(r)}&<\sum_{j>j_0}4^{-j}\sqrt{\frac{27}{16\pi j}}+(n/6)4^{-n/6}\sqrt{\frac{27}{16\pi n/6}}2^{n/7+1}\\
&\leq4^{-j_0}\sqrt{\frac{3}{\pi j_0}}+\sqrt{\frac{9n}{32\pi}}2^{n/7-n/3+1}\\
&\leq n^{-2}\sqrt{\frac{3\log 2}{\pi \log n}}+2^{-4n/21}\sqrt{\frac{9n}{8\pi}}.
\end{align*}
Applying analogous arguments, and by using the inequality $\frac{6j+1}{3j+1}<2$, we get
\begin{align*}
\sum_{j>j_0}\frac{\tilde{E}_{n,j}(r)}{E_{n,0}(r)}&<n^{-2}\sqrt{\frac{12\log 2}{\pi \log n}}+2^{-4n/21}\sqrt{\frac{9n}{2\pi}}.
\end{align*}

Finally, using the fact that
$$g_P(n,r)\leq g_P(n,1)=\sum_{k=2}^{n-1}\frac23k^2=\frac{2n^3-3n^2+n-6}{9}<\frac{2}{9}n^3$$
for $n\geq2$, we conclude that
$$
\varepsilon_{3,D}(n,r)\leq \sqrt{\frac{4\log 2}{3n \log n}}+2^{-1/2-4n/21}n^2,$$
and
$$
\varepsilon_{3,E}(n,r)\leq \sqrt{\frac{16\log 2}{3n \log n}}+2^{1/2-4n/21}n^2.
$$
We finish the proof by using the condition $n>7000$ in the above bounds, and
by recalling \eqref{eqFtildeRatio} to draw a similar conclusion about $\varepsilon_{3,F}(n,r)$.
\end{proof}

\section{Bounding the tails}\label{seEps2}
In order to bound the error term $\epsilon_{2,P}(n,r)$, we need bounds on $\frac{P_{n,j}(re^{i\theta})}{\tilde{P}_{n,j}(r)}$ as well as on $\frac{\tilde{P}_{n,j}(r)}{\tilde{P}_{n,0}(r)}$. The results of previous section, along with \eqref{eqFtildeRatio}, imply the inequalities
\begin{align}
\sum_{j=0}^{j_0}\frac{\tilde{D}_{n,j}(r)}{\tilde{D}_{n,0}(r)}&<\sum_{j=0}^{\infty}3^{-3j}\binom{3j+1}{j}<1.185,\label{eqDTailFactor1}\\
\sum_{j=0}^{j_0}\frac{\tilde{E}_{n,j}(r)}{\tilde{E}_{n,0}(r)}&<\sum_{j=0}^{\infty}3^{-3j}\binom{3j+1}{j}1.003^j\frac{6j+1}{3j+1}<1.329,\label{eqETailFactor1} \\
\sum_{j=0}^{j_0}\frac{\tilde{F}_{n,j}(r)}{\tilde{F}_{n,0}(r)}&<\sum_{j=0}^{j_0}\frac{\tilde{E}_{n,j}(r)}{\tilde{E}_{n,0}(r)}<1.329. \label{eqFTailFactor1}
\end{align}

We now turn our attention to the quotient $\frac{P_{n,j}(re^{i\theta})}{\tilde{P}_{n,j}(r)}$.

The definitions of $\tilde{P}_{n,j}(r)$ implies that
\[
\left|\frac{P_{n,j}(re^{i\theta})}{\tilde{P}_{n,j}(r)}\right|\leq \prod_{m=3j+2}^{n-3j-1}\left|\frac{1+r^me^{im\theta}+r^{2m}e^{2im\theta}}{1+r^m+r^{2m}}\right|.
\]

The next lemma provides bounds on the factors on the right-hand side.

\begin{lemma}\label{leIneqTailAux}
For all $r\in\mathbb{R}^+$ and $\theta\in\mathbb{R}$, we have
\begin{equation}
\left|\frac{1+re^{i\theta}+r^2e^{2i\theta}}{1+r+r^2}\right|\leq \exp\left(-\frac{2r}{1+r^2} \sin^2(\theta/2)\right).
\end{equation}
\end{lemma}

\begin{proof}
It is straightforward to calculate
\begin{align*}
\left|1+re^{i\theta}+r^2e^{2i\theta}\right|^2&=(1+re^{i\theta}+r^2e^{2i\theta})(1+re^{-i\theta}+r^2e^{-2i\theta})\\
&=1+2r\cos\theta+(4\cos^2\theta-1)r^2+2r^3\cos\theta+r^4\\
&=1+(2-4s)r+(3-16s+16s^2)r^2+(2-4s)r^3+r^4,
\end{align*}
where $s=\frac12(1-\cos\theta)=\sin(\theta/2)^2\in[0,1].$

We claim that
\[
\left|\frac{1+re^{i\theta}+r^2e^{2i\theta}}{1+r+r^2}\right|^2\leq\left(\frac{1-rs+r^2}{1+rs+r^2}\right)^2\leq \exp\left(-\frac{4rs}{1+r^2}\right).
\]
The first inequality is proved by the algebraic manipulation
\begin{multline*}
(1+r+r^2)^2(1-rs+r^2)^2
-(1+rs+r^2)^2\left(1+(2-4s)r+(3-16s+16s^2)r^2+(2-4s)r^3+r^4\right)\\
=4r^2s(1-s)\left((1+r^2)(2-r+r^2+7rs)+4r^2s^2\right)\geq0,
\end{multline*}
while the second inequality can be obtained by taking $x=\frac{2rs}{1+r^2}$ in the inequality $\frac{1-x/2}{1+x/2}\leq e^{-x}$, which holds for all $x\in[0,1]$.
\end{proof}

%From the expressions \eqref{eqAnjFactor1} and \eqref{eqAnjTilde1} we can conclude that
%\[
%\left|\frac{A_{n,j}(re^{i\theta})}{\tilde{A}_{n,j}(r)}\right|\leq\prod_{k=3j+1}^{n-3j}\left|\frac{1+r^ke^{ik\theta}+r^{2k}e^{2ik\theta}}{1+r^k+r^{2k}}\right|,
%\]
%so we first bound each factor by using the lemma below.

An immediate consequence of Lemma~\ref{leIneqTailAux} is the inequality
\begin{align}
-\log \left|\frac{P_{n,j}(re^{i\theta})}{\tilde{P}_{n,j}(r)}\right|&\geq\sum_{m=3j+2}^{n-3j-1}\frac{r^{m}}{1+r^{2m}} \sin(m\theta/2)^2\nonumber \\
&\geq \frac{r^{3j+2}}{1+r^{6j+4}}\sum_{m=3j+2}^{n-3j-1} r^{m-3j-2}\sin(m\theta/2)^2. \label{eqTailRatio1}
\end{align}

In the lemma below, we provide an inequality concerning the last sum above.

\begin{lemma}\label{leIneqSineSum}
Let $a,b\in\Z^+$ such that $b\geq2$, and $r\in[0,1]$. Then we have
\[
\sum_{m=a}^{a+b-1}r^{m-a} \sin(m\theta/2)^2\geq\frac12\frac{1-r^b}{1-r}\left(1-\sqrt{\frac{1+\kappa\frac{(1+r^b)^2}{(1-r^b)^2}\tan^2\frac\theta2}{1+\kappa\frac{(1+r)^2}{(1-r)^2}\tan^2\frac\theta2}}\right),
\]
where
\[
\kappa=\frac{(1-r^b)(1-r^{b/3})}{(1+r^b)(1+r^{b/3})}.
\]
\end{lemma}

\begin{proof}
This sum has a closed form,
\begin{align*}
\sum_{m=a}^{a+b-1}r^{m-a}& \sin(m\theta/2)^2\\
&=\frac{1}{2}\left(\frac{1-r^b}{1-r}-\frac{(\cos a\theta-r\cos((a-1)\theta))-r^b(\cos (a+b)\theta-r\cos((a+b-1)\theta))}{1-2r\cos\theta+r^2}\right)\\
&=\frac{1}{2}\left(\frac{1-r^b}{1-r}-\frac{(\cos a\theta-r^b\cos (a+b)\theta)(1-r\cos\theta)-(\sin a\theta-r^b\sin (a+b)\theta)\sin\theta}{1-2r\cos\theta+r^2}\right).
\end{align*}
We use the Cauchy--Schwarz inequality, observe that
\[
(1-r\cos\theta)^2+(r\sin\theta)^2=1-2r\cos\theta+r^2\]
and
\[
(\cos a\theta-r^b\cos (a+b)\theta)^2+(\sin a\theta-r^b\sin (a+b)\theta)^2=1-2r^{b}\cos b\theta+r^{2b},
\]
to arrive at
\begin{equation}\label{26a}
\sum_{m=a}^{b-1}r^{m-a} \sin(m\theta/2)^2\geq\frac{1}{2}\left(\frac{1-r^b}{1-r}-\sqrt\frac{1-2r^{b}\cos b\theta+r^{2b}}{1-2r\cos\theta+r^2}\right).
\end{equation}

Comparing \eqref{26a} with the claims of this lemma, we see that
it suffices to prove that
\[
\frac{1-2r^{b}\cos b\theta+r^{2b}}{1-2r\cos\theta+r^2}\leq \frac{(1-r^b)^2+\kappa(1+r^b)^2\tan^2\frac\theta2}{(1-r)^2+\kappa(1+r)^2\tan^2\frac\theta2}.
\]
By routine manipulation, the above inequality is equivalent to
\begin{equation}\label{26b}
\cos\theta-\cos b\theta\leq\frac{(1-\kappa)(1-r^{b-1})(1-r^{b+1})\sin^2\theta}{r^{b-1}\left((1-r)^2(1+\cos\theta)+\kappa(1+r)^2(1-\cos\theta)\right)}.
\end{equation}
Here, Lemma~\ref{leIneqChebyshev} implies the inequality
\begin{equation}\label{26c}
\cos\theta-\cos b\theta\leq \frac{(b^2-1)\sin^2\theta}{(1+\cos\theta)+b^2(1-\cos\theta)/3}.
\end{equation}

Comparing \eqref{26b} and \eqref{26c}, we see that it remains to show that
\begin{equation}\label{26d}
\frac{r^{b-1}\left((1-r)^2(1+\cos\theta)+\kappa(1+r)^2(1-\cos\theta)\right)}{(1-\kappa)(1-r^{b-1})(1-r^{b+1})}\leq \frac{(1+\cos\theta)+b^2(1-\cos\theta)/3}{b^2-1}.
\end{equation}
This is an immediate consequence of Lemma~\ref{leIneqRB1} and the inequality
\begin{equation}\label{26f}
\kappa\leq \frac{b^2}{3}\frac{(1-r)^2}{(1+r)^2}.
\end{equation}
Equation \eqref{26f} can be directly verified for $b=2$. If $b\geq3$, we write $r=e^{-x/2}$, so that the inequality is equivalent to
\[
\frac{\tanh (bx)\tanh (bx/3)}{(\tanh x)^2}\leq\frac{b^2}{3}.
\]
This follows finally from the fact that $\tanh x/x$ is decreasing on $\R^+$.
\end{proof}

Now we have all the tools in order to prove an upper bound for the quotient $\left|\frac{P_{n,j}(re^{i\theta})}{\tilde{P}_{n,j}(r)}\right|$.
\begin{prop}\label{leIneqTails1}
For all $n>32$, $r\in(0,1]$, $\theta\in[-\pi,\pi]$ and $0\leq j\leq j_0$, we have
\[
\left|\frac{P_{n,j}(re^{i\theta})}{\tilde{P}_{n,j}(r)}\right|<\exp\left(-\phi(n,3j_0+2,n-6j_0-2,r,\rho)\right),
\]
where $\rho=\theta\frac{1-r^n}{1-r}$, and
\[
\phi(n,a,b,r,\rho):=\frac{b^3}{n^3}\frac{r^{a}}{1+r^{2a}}\frac{(1+r)^2}{4}\frac{1-r^{n/12}}{1-r}\left(1-\sqrt\frac{1+\frac{(1-r)^2(1+r^b)^2}{(1+r)^2(1-r^b)^2}\rho^2}{1+\rho^2}\right).
\]
\end{prop}

\begin{proof}
The condition $n>32$ ensures that $n>6j_0+2$. Therefore, Lemmas~\ref{leIneqTailAux} and \ref{leIneqSineSum} imply that
\begin{align*}
-\log\left|\frac{P_{n,j}(re^{i\theta})}{\tilde{P}_{n,j}(r)}\right|&\geq\sum_{m=3j+2}^{n-3j-1}\log\left|\frac{1+r^me^{im\theta}+r^{2m}e^{2im\theta}}{1+r^m+r^{2m}}\right|\\
&\geq\sum_{m=3j_0+2}^{n-3j_0-1}\frac{2r^m}{1+r^{2m}} \sin(m\theta/2)^2\\
&\geq\frac{2r^a}{1+r^{2a}}\sum_{m=a}^{a+b-1}r^{m-a} \sin(m\theta/2)^2,
\end{align*}
where we write $a=3j_0+2$ and $b=n-6j_0-2$ for simplicity of notation.

Now Lemma~\ref{leIneqSineSum} allows us to do further estimation:
\[
-\log\left|\frac{P_{n,j}(re^{i\theta})}{\tilde{P}_{n,j}(r)}\right|\geq\frac{r^a}{1+r^{2a}}\frac{1-r^b}{1-r}\left(1-\sqrt{\frac{1+\kappa\frac{(1+r^b)^2}{(1-r^b)^2}\tan^2\frac\theta2}{1+\kappa\frac{(1+r)^2}{(1-r)^2}\tan^2\frac\theta2}}\right),
\]
where
\[
\kappa=\frac{(1-r^b)(1-r^{b/3})}{(1+r^b)(1+r^{b/3})}.
\]
After substituting $\theta=\rho\frac{1-r}{1-r^n}$, we first note
\[
\tan\frac\theta2>\frac\theta2=\frac{\rho}{2}\frac{1-r}{1-r^n},
\]
valid for $|\theta|\leq\pi$.
Then we use the fact that $\frac{1+cx}{1+cy}$ is decreasing with respect to $c$ if $y>x>0$ to estimate the factor in terms of $\rho$:
\[
1-\sqrt{\frac{1+\kappa\frac{(1+r^b)^2}{(1-r^b)^2}\tan^2\frac\theta2}{1+\kappa\frac{(1+r)^2}{(1-r)^2}\tan^2\frac\theta2}}> 1-\sqrt{\frac{1+\kappa\frac{(1+r^b)^2(1-r)^2}{4(1-r^b)^2(1-r^{n})^2}\rho^2}{1+\kappa\frac{(1+r)^2}{4(1-r^{n})^2}\rho^2}}.
\]

By exploiting the inequality
\[
1-\sqrt\frac{1+cx}{1+cy}\geq c\left(1-\sqrt\frac{1+x}{1+y}\right)
\]
for all $0<c\leq1$ and $y>x>0$, and by taking
\[
c=\kappa\frac{(1+r)^2}{4(1-r^{n})^2}=\frac{(1+r)^2}{4}\frac{1-r^{b/3}}{1-r^b}\left(\frac{1-r^b}{1-r^n}\right)^2\frac{1}{(1+r^b)(1+r^{b/3})}\leq1,
\]
we arrive at the expected result:
\begin{align*}
-\log\left|\frac{P_{n,j}(re^{i\theta})}{\tilde{P}_{n,j}(r)}\right|&\geq\frac{r^{a}}{1+r^{2a}}\frac{1-r^{b/3}}{1-r}\left(\frac{1-r^b}{1-r^n}\right)^2\frac{(1+r)^2}{4(1+r^b)(1+r^{b/3})}\left(1-\sqrt\frac{1+\frac{(1-r)^2(1+r^b)^2}{(1+r)^2(1-r^b)^2}\rho^2}{1+\rho^2}\right)\\
&=\frac{r^{a}}{1+r^{2a}}\frac{1-r^{b/12}}{1-r}\left(\frac{1-r^b}{1-r^n}\right)^2\frac{(1+r)^2(1+r^{b/6})(1+r^{b/12})}{4(1+r^b)(1+r^{b/3})}\\
&\kern5cm
\times
\left(1-\sqrt\frac{1+\frac{(1-r)^2(1+r^b)^2}{(1+r)^2(1-r^b)^2}\rho^2}{1+\rho^2}\right)\\
&\geq\frac{b^3}{n^3}\frac{r^{a}}{1+r^{2a}}\frac{1-r^{n/12}}{1-r}\frac{(1+r)^2}{4}\left(1-\sqrt\frac{1+\frac{(1-r)^2(1+r^b)^2}{(1+r)^2(1-r^b)^2}\rho^2}{1+\rho^2}\right),
\end{align*}
where the last step uses the inequality $\frac{1-r^x}{1-r^y}>\frac{x}{y}$, valid for $0\leq r\leq1$ and $y \geq x$.
\end{proof}

In order to convert the above lemma into an upper bound for $\epsilon_{2,P}(n,r)$, we first note that\break $\phi(n,a,b,r,\rho)$ is increasing with respect to $\rho$. We estimate the integral in the definition of $\epsilon_{2,P}(n,r)$ by making the substitution $\theta=\frac{1-r}{1-r^n}\rho$ as in Proposition~\ref{leIneqTails1} and by splitting the integral at $\rho=\frac{3}{2}$, as shown below:
\begin{align*}
\int^{2\pi-\theta_0}_{\theta_0}&\sup_{0\leq j\leq j_0}\left|\frac{P_{n,j}(re^{i\theta})}{\tilde{P}_{n,j}(r)}\right|\,d\theta
\leq2\frac{1-r}{1-r^n}\int_{1/3}^{\pi\frac{1-r^n}{1-r}}\exp\left(-\phi(n,3j_0+2,n-6j_0-2,r,\rho)\right)d\rho\\
&\leq2\frac{1-r}{1-r^n}\left(\int_{1/3}^{3/2}+\int_{3/2}^{\pi\frac{1-r^n}{1-r}}\right)\exp\left(-\phi(n,3j_0+2,n-6j_0-2,r,\rho)\right)d\rho\\
&<2\frac{1-r}{1-r^n}\int_{1/3}^{3/2}\exp\left(-\phi(n,3j_0+1,n-6j_0,r,\rho)\right)d\rho\\
&\kern5cm
+2\pi \exp\left(-\phi(n,3j_0+2,n-6j_0-2,r,3/2)\right).
\end{align*}

Suppose for now that $n>7000$ and $r\in(r_0,1]$. By looking at various factors in the definition of $\phi(n,a,b,r,\rho)$, we observe that
\begin{align*}
\frac{n-6j_0-2}{n}&\geq1-\frac{6\log_2n+2}{n}>0.9887,\\
r^{3j_0+2}&\geq \exp\left(-(3\log_2n+2)\sqrt{\alpha/n}\right)>0.5958,\\
\frac{(1+r)^2}{4}&\geq\frac{(1+r_0)^2}{4}>\frac{74}{75},\\
\frac{(1-r)(1+r^b)}{(1+r)(1-r^b)}&\leq \frac{(1-r_0)(1+r_0^n)}{(1+r_0)(1-r_0^n)}<\frac{1}{150}.
\end{align*}
These observations enable us to conclude
\begin{align*}
\phi(n,3j_0+2,n-6j_0-2,r,\rho)&>0.9887^3\frac{0.5958}{1+(0.5958)^2}\frac{74}{75}\left(1-\sqrt\frac{1+(\rho/150)^2}{1+\rho^2}\right)\frac{1-r^{n/12}}{1-r}\\
&>\frac{5}{12}\left(1-\sqrt\frac{1+10^{-4}}{1+\rho^2}\right)\frac{1-r^{n/12}}{1-r},
\end{align*}
for all $n>7000$ and $\rho\in[1/3,3/2]$. We define
\begin{equation}\label{eqDefPhiStar}
\phi^*(n,r,\rho):=\frac{5}{12}\left(1-\sqrt\frac{1+10^{-4}}{1+\rho^2}\right)\frac{1-r^{n/12}}{1-r},
\end{equation}
and obtain that
\begin{align}
\int^{2\pi-\theta_0}_{\theta_0}&\sup_{0\leq j\leq j_0}\left|\frac{P_{n,j}(re^{i\theta})}{\tilde{P}_{n,j}(r)}\right|\,d\theta\nonumber\\
\nonumber
&<2\frac{1-r}{1-r^n}\int_{1/3}^{3/2}\exp\left(-\phi(n,3j_0+1,n-6j_0,r,\rho)\right)d\rho\\
&\kern5cm
+2\pi \exp\left(-\phi(n,3j_0+2,n-6j_0-2,r,3/2)\right)\nonumber\\
&<2\frac{1-r}{1-r^n}\int_{1/3}^{3/2}\exp\left(-\phi^*(n,r,\rho)\right)d\rho+2\pi \exp\left(-\phi^*(n,r,3/2)\right).\label{eqPRatioTail1}
\end{align}

At this point, we incorporate the factor $\sqrt{g_P(n,r)}$ in the definition of $\epsilon_{2,P}(n,r)$. We note that, using \eqref{eqGUpperBound} and \eqref{eqIneqR000_2}, we have
\begin{equation}\label{eqDefGStar}
g_P(n,r)<\frac{12}{5}\left(\frac{1-r^n}{1-r}\right)^3.
\end{equation}
In view of this upper bound, we prove some related monotonicity results.
\begin{lemma}\label{leTailMono}
Let $\phi^*$ be defined as in \eqref{eqDefPhiStar}. For all $n>7000$ and all $r\in(r_0,1]$, we have:
\begin{itemize}
\item The function
\begin{equation}\label{eqMono1}
\left(\frac{1-r^n}{1-r}\right)^{3/2}\exp\left(-\phi^*(n,r,3/2)\right)
\end{equation}
is decreasing with respect to $r$.
\item If $\rho\in[1/3,3/2]$, then the function
\begin{equation}\label{eqMono2}
\left(\frac{1-r^n}{1-r}\right)^{1/2}\exp\left(-\phi^*(n,r,\rho)\right)
\end{equation}
is also decreasing with respect to $r$.
\end{itemize}
\end{lemma}

\begin{proof}
By taking logarithmic derivatives with respect to $r$, these claims are equivalent to the inequalities
\begin{equation}
\frac32\frac{\partial}{\partial r}\log \frac{1-r^n}{1-r}\leq\frac{5}{12}\left(1-\sqrt\frac{1+10^{-4}}{1+(3/2)^2}\right)\frac{\partial}{\partial r}\frac{1-r^{n/12}}{1-r}, \label{eqMonoIneq1}
\end{equation}
and
\begin{equation}
\frac12\frac{\partial}{\partial r}\log \frac{1-r^n}{1-r}\leq\frac{5}{12}\left(1-\sqrt\frac{1+10^{-4}}{1+\rho^2}\right)\frac{\partial}{\partial r}\frac{1-r^{n/12}}{1-r}.\label{eqMonoIneq2}
\end{equation}
In order to prove \eqref{eqMonoIneq1} and \eqref{eqMonoIneq2}, we perform the following calculations:
\begin{itemize}
\item Lemmas~\ref{leRDerivative1} and \ref{leRRational1} imply that
\[
\frac{\partial}{\partial r}\frac{1-r^{n/12}}{1-r}\geq \frac{(1-r^{n/12})(1-r^{(n-12)/24})}{(1-r)^2}\geq24\frac{1-r^n}{1-r}.
\]
\item Again, Lemma~\ref{leRDerivative1} imply that
\begin{align*}
\frac{\partial}{\partial r}\log \frac{1-r^n}{1-r}&\leq\frac{1-r^n}{1-r}.
\end{align*}
\item We have
\[
\frac{5}{12}\left(1-\sqrt\frac{1+10^{-4}}{1+(3/2)^2}\right)\approx 0.18553>\frac16,
\]
therefore the right-hand side of \eqref{eqMonoIneq1} is at least $4\frac{1-r^n}{1-r}$.
\item Since $\rho\geq 1/3$, we have
\[
\frac{5}{12}\left(1-\sqrt\frac{1+10^{-4}}{1+\rho^2}\right)\geq \frac{5}{12}\left(1-\sqrt\frac{1+10^{-4}}{1+1/9}\right)\approx0.021362>\frac{1}{48},
\]
and thus the right-hand side of \eqref{eqMonoIneq2} is at least $\frac12\frac{1-r^n}{1-r}$. \qedhere
\end{itemize}
\end{proof}

We are now ready to provide explicit upper bounds for $\epsilon_{2,P}(n,r)$.

\begin{lemma}\label{leIneqTail}
Suppose that $n>n_0=7000$, and that $r_0$, $j_0$ and $\theta_0$ are as defined as in \eqref{eqCutoffR0}, \eqref{eqCutoffJ0} and \eqref{eqCutoffTheta0}, respectively. Then, for all $r\in(r_0,1]$, we have
\begin{align*}
\epsilon_{2,D}(n,r)&<0.237, & \epsilon_{2,E}(n,r)&<0.266, &\epsilon_{2,F}(n,r)&<0.266.
\end{align*}
\end{lemma}
\begin{proof}
Making use of Lemmas~\ref{leIneqTails1} and \ref{leTailMono} as well as of \eqref{eqPRatioTail1} and \eqref{eqDefGStar}, and also noticing that $\phi^*(n,r,\rho)$ is increasing with respect to $n$, we infer
\begin{align*}
\epsilon_{2,P}(n,r)&=\frac{\sqrt{g_P(n,r)}}{\sqrt{2\pi}}\left(\sum_{j=0}^{j_0}\frac{\tilde{P}_{n,j}(r)}{\tilde{P}_{n,0}(r)}\right)\left(\int^{2\pi-\theta_0}_{\theta_0}\sup_{0\leq j\leq j_0}\left|\frac{P_{n,j}(re^{i\theta})}{\tilde{P}_{n,j}(r)}\right|\,d\theta\right)\\
&< \sqrt{\frac{24}{5\pi}}\left(\frac{1-r^n}{1-r}\right)^{3/2}\left(\sum_{j=0}^{j_0}\frac{\tilde{P}_{n,j}(r)}{\tilde{P}_{n,0}(r)}\right)\\
&\kern1cm
\times
\left(\frac{1-r}{1-r^n}\int_{1/3}^{3/2}\exp\left(-\phi^*(n,r,\rho)\right)d\rho+\pi \exp\left(-\phi^*(n,r,3/2)\right)\right)\\
&=\sqrt{\frac{24}{5\pi}}\left(\sum_{j=0}^{j_0}\frac{\tilde{P}_{n,j}(r)}{\tilde{P}_{n,0}(r)}\right)\\
&\kern1cm
\times
\left(\left(\frac{1-r^n}{1-r}\right)^{1/2}\int_{1/3}^{3/2}\exp\left(-\phi^*(n,r,\rho)\right)d\rho+\pi \left(\frac{1-r^n}{1-r}\right)^{3/2}\exp\left(-\phi^*(n,r,3/2)\right)\right)\\
&<\sqrt{\frac{24}{5\pi}}\left(\sum_{j=0}^{j_0}\frac{\tilde{P}_{n,j}(r)}{\tilde{P}_{n,0}(r)}\right)\\
&\kern1cm
\times
\left(\left(\frac{1-r_0^n}{1-r_0}\right)^{1/2}\int_{1/3}^{3/2}\exp\left(-\phi^*(n,r_0,\rho)\right)d\rho+\pi \left(\frac{1-r_0^n}{1-r_0}\right)^{3/2}\exp\left(-\phi^*(n,r_0,3/2)\right)\right)\\
&<\sqrt{\frac{24}{5\pi}}\left(\sum_{j=0}^{j_0}\frac{\tilde{P}_{n,j}(r)}{\tilde{P}_{n,0}(r)}\right)\\
&\kern.7cm
\times
\left(\left(\frac{1}{1-r_0}\right)^{1/2}\int_{1/3}^{3/2}\exp\left(-\phi^*(n_0,r_0,\rho)\right)d\rho+\pi \left(\frac{1}{1-r_0}\right)^{3/2}\exp\left(-\phi^*(n_0,r_0,3/2)\right)\right).
\end{align*}
Now we substitute $n_0=7000$ and $r_0=\exp(-\sqrt{\alpha/n_0})\approx0.987239$, and observe that $\frac{1}{1-r_0}\approx78.3612$ and $\phi^*(n_0,r_0,3/2)\approx14.5302$.
Moreover, we use numerical integration to calculate
\[
\int_{1/3}^{3/2}\exp\left(-\phi^*(n_0,r_0,\rho)\right)d\rho\approx0.0177756<\frac{4}{225}.
\]
Therefore, we infer that
\begin{align*}
\sqrt{\frac{24}{5\pi}}\left(\frac{1}{1-r_0}\right)^{1/2}\int_{1/3}^{3/2}&\exp\left(-\phi^*(n,r_0,\rho)\right)d\rho+\pi \left(\frac{1}{1-r_0}\right)^{3/2}\exp\left(-\phi^*(n,r_0,3/2)\right)\\
&<\sqrt{\frac{24}{5\pi}}\left(78.3612^{1/2}\times\frac{4}{225}+78.3612^{3/2}\pi\times e^{-14.5302}\right)\\
&\approx0.195842<\frac15.
\end{align*}
If this inequality is combined with \eqref{eqDTailFactor1} and \eqref{eqETailFactor1}, the proof is complete.
\end{proof}

\section{Concluding the Proof}\label{seMain}
Having finally obtained upper bounds for all the error terms, we combine them to derive the main result of this paper.

\begin{theorem}\label{thMain}
The Borwein Conjecture is true for all $n> n_0=7000$.
\end{theorem}
\begin{proof}
When $n> n_0$, we can see that $\lambda=\frac {r-r^{n+1}} {1-r}>\frac{e^{\sqrt{\alpha/n_0}-e^{(n_0+1)\sqrt{\alpha/n_0}}}}{1-e^{\sqrt{\alpha/n_0}}}>77.$ Thus, from Lemma~\ref{leIneqMainTerm} we infer
\[
\epsilon_{0,P}(n,m,r)<\frac{7\sqrt{2}}{\sqrt{3\pi\lambda}}+\erfc\sqrt\frac{\lambda}{84}<\frac{7\sqrt{2}}{\sqrt{231\pi}}+\erfc\sqrt\frac{77}{84}\approx0.54321<0.544.
\]
Also, $\lambda>77$ allows us to conclude that
\[
\left(1+\frac{\sqrt{5}}{3\sqrt{3\lambda}}\right)<1.05,
\]
which results in explicit bounds for $\epsilon_{1,P}(n,r)$ in Lemma~\ref{leIneqSmallJPeak},
\begin{align*}
\epsilon_{1,D}(n,r)&<0.187\times1.05<0.197, \\
 \epsilon_{1,E}(n,r)&<0.043\times1.05<0.046, \\
\epsilon_{1,F}(n,r)&<0.043\times1.05<0.046.
\end{align*}

Remembering the estimations in Lemmas~\ref{leIneqLargeJ} and \ref{leIneqTail}, we make a table of the upper bounds we have obtained so far:
\begin{table}[!htbp]
  \centering
  \begin{tabular}{|c|c|c|c|c|c|}
    \hline
    $P$ & $\epsilon_{0,P}\leq$ & $\epsilon_{1,P}\leq$ & $\epsilon_{2,P}\leq$ & $\epsilon_{3,P}\leq$ & Sum \\
    \hline
%    $A$ &  & 0.139 & 0.228 & 0.003 & 0.914 \\
    $D$ & \multirow{3}{*}{0.544} & 0.197 & 0.237 & 0.004 & 0.982 \\
    $E$ &                        & 0.046 & 0.266 & 0.008 & 0.864 \\
    $F$ &                        & 0.046 & 0.266 & 0.008 & 0.864 \\
    \hline
  \end{tabular}
  \caption{List of upper bounds for the quantities $\epsilon_{i,P}(n,r)$.}\label{tbBounds}
\end{table}

From this table we can finally conclude that
\[
\epsilon_{0,P}(n,m,r)+\epsilon_{1,P}(n,r)+\epsilon_{2,P}(n,r)+\epsilon_{3,P}(n,r)<1
\]
holds for all $P\in\{D,E,F\}$ and $n> n_0$, confirming the truth of the Borwein Conjecture in this range.
\end{proof}

\section{Computer verification for $n\leq7000$}\label{seVerify}

We have explicitly verified $[q^m]P_n(q)>0$ for all $P\in\{A,B,C\}$, and all
$n$ and $m$ with $1\leq n\leq7000$ and $0\leq m\leq n^2$ by using a computer. The program consists of calculating and checking the coefficients of $\frac{(q;q)_{3n}}{(q^3;q^3)_n}$ by repeated polynomial multiplication, using the GMP library \cite{GMP} for exact large-integer arithmetic. The computation was run at Johannes Kepler University in Linz, on a computer with 32 Intel Xeon processors at 2GHz (of which only 10 are used). The running time was 53 hours, and used up to 150 gigabytes of memory for storing all the coefficients.

\section{Discussion}\label{seDiscuss}
There are two more Borwein Conjectures mentioned in \cite{MR1395410}: a ``Second Borwein Conjecture" that also relates to modulus 3, and a ``Third Borwein Conjecture" that relates to modulus 5.

\begin{conj}[\sc P. Borwein]\label{cjBorwein2}
Let the polynomials $\alpha_n(q)$, $\beta_n(q)$ and $\gamma_n(q)$ be defined by the relationship
\begin{equation}\label{eqConjecture2}
\frac{(q;q)^2_{3n}}{(q^3;q^3)^2_n}=\alpha_n(q^3)-q\beta_n(q^3)-q^2\gamma_n(q^3).
\end{equation}
Then these polynomials have non-negative coefficients.
\end{conj}

\begin{conj}[\sc P. Borwein]\label{cjBorwein3}
Let the polynomials $\nu_n(q)$, $\phi_n(q)$, $\chi_n(q)$, $\psi_n(q)$ and $\omega_n(q)$ be defined by the relationship
\begin{equation}\label{eqConjecture3}
\frac{(q;q)_{5n}}{(q^5;q^5)_n}=\nu_n(q^5)-q\phi_n(q^5)-q^2\chi_n(q^5)-q^3\psi_n(q^5)-q^4\omega_n(q^5),
\end{equation}
Then these polynomials have non-negative coefficients.
\end{conj}
Both these conjectures are still wide open. In particular, no reasonable formulas for the polynomials have been found so far. We remark that the comparison of \eqref{eqConjecture} and \eqref{eqConjecture2} yields the relationship $\alpha_n(q)=A_n^2(q)+2qB_n(q)C_n(q)$, so non-negativity for the coefficients of $\alpha_n(q)$ follows trivially from this paper.

Recall that for our proof we used the formulas for $A_n(q)$, $B_n(q)$ and $C_n(q)$ given in Theorem~\ref{thAndrewsExpansion}. As we mentioned, these formulas had apparently not caught much attention so far. It is rather a different type of formula that was found to be much more inspiring, namely (see \cite[Theorem~3.1]{MR1395410})
\begin{align}
A_n(q)&=\sum_{j\in\Z}(-1)^jq^{j(9j+1)/2}\qbinom{2n}{n+3j}_q, \label{eqExpansionOldA} \\
B_n(q)&=\sum_{j\in\Z}(-1)^jq^{j(9j-5)/2}\qbinom{2n}{n+3j-1}_q, \label{eqExpansionOldB} \\
C_n(q)&=\sum_{j\in\Z}(-1)^jq^{j(9j+7)/2}\qbinom{2n}{n+3j+1}_q, \label{eqExpansionOldC}
\end{align}
where we used again the standard notation for $q$-binomial coefficients. These are so much more imaginative because of their resemblance with a family of formulas appearing as generating functions for partitions with restricted hook differences in \cite{MR930170}. Andrews et al.\ had shown that
\begin{equation}
\sum_{j\in\Z}(-1)^jq^{jK\frac{j(\alpha+\beta)+\alpha-\beta}{2}}\begin{bmatrix}m+n \\ n-Kj\end{bmatrix}_q \label{eqGenHookDiff}
\end{equation}
is the generating function for certain partitions with restricted hook differences, with $\alpha,\beta,K,m,n$ being non-negative integers satisfying $\alpha+\beta<2K$ and $\beta-K\leq n-m\leq K-\alpha$. Indeed, the generating function in \eqref{eqExpansionOldA} is the ``special case" of \eqref{eqGenHookDiff} in which $m=n$, $\alpha=5/3$, $\beta=4/3$ and $K=3$. Similar observations hold for $B_n(q)$ and $C_n(q)$. In other words, the result of Andrews et al.\ seems to produce a proof of the Borwein Conjecture, except for the small flaw that the choices of $\alpha$ and $\beta$ are not integral, and thus not legal.

%\begin{theorem}[\sc Andrews et al. {\cite{MR930170}}]
%Let $K$ be a positive integer, and $m,n,\alpha,\beta$ be non-negative integers, satisfying $\alpha+\beta<2K$ and $\beta-K\leq n-m\leq K-\alpha$. Then the polynomial
%\[
%G(m,n;\alpha,\beta,K;q):=\sum_{j\in\Z}(-1)^jq^{jK\frac{j(\alpha+\beta)+\alpha-\beta}{2}}{m+n \brack n-Kj}_q
%\]
%is the generating function for partitions inside a $m\times n$ rectangle that satisfy some so-called "hook difference conditions" specified by $\alpha,\beta$ and $K$.
%\end{theorem}

Bressoud \cite{MR1392489} extended the mystery by making the following much more general conjecture.
\begin{conj}[\sc {Bressoud \cite[Conjecture 6]{MR1392489}}]
Suppose that $m,n\in\Z^+$, $\alpha$ and $\beta$ are positive rational numbers, and $K$ is a positive integer such that $\alpha K$ and $\beta K$ are integers. If $1\leq\alpha+\beta\leq2K+1$ (with strict inequalities if $K=2$) and $\beta-K\leq n-m\leq K-\alpha$, then the polynomial
\[
\sum_{j\in\Z}(-1)^jq^{j(K(\alpha+\beta)j+K(\alpha-\beta))/2}\qbinom{m+n}{m+Kj}
\]
has non-negative coefficients.
\end{conj}
To this day, Bressoud's conjecture has only been proved when $\alpha,\beta\in\Z$ (corresponding to the result of Andrews et al.~\cite{MR930170} mentioned above), and some sporadic cases where the denominator of either $\alpha$ or $\beta$ is a power of 2 (see \cite{MR1874535,MR2009544}). The connection to partitions with hook difference conditions lets one hope that a similar combinatorial interpretation may exist for the polynomials in the Borwein Conjecture, but to this day no such connection has been found.

Our approach for proving Theorem~\ref{thMain} has been analytic. The formulas that we just discussed, in particular the formulas \eqref{eqExpansionOldA}--\eqref{eqExpansionOldC} for $A_n(q)$, $B_n(q)$ and $C_n(q)$, are unsuitable for asymptotic approximation. The reason is that each dominating term in the sums \eqref{eqExpansionOldA}--\eqref{eqExpansionOldC} has order $O(4^n/n)$, whereas the actual order of magnitude of $A_n(q)$, $B_n(q)$ and $C_n(q)$ is trivially bounded above by $O(3^n)$. In other words, in the sums \eqref{eqExpansionOldA}--\eqref{eqExpansionOldC}, there is a huge amount of cancellations going on, which are seemingly impossible to control in order to find reasonable asymptotic estimates. As opposed to that, only the first term in the formulas in Theorem~\ref{thAndrewsExpansion} contributes to the sum, as the other terms are asymptotically negligible, as we have shown.

We also mention the result of Li \cite{1512_01191}, which proves the positivity of the sum
\[
\sum_{m\equiv k\ (\text{mod }{n+1})}[q^m]A_n(q)
\]
for all $k$ with $0\leq k\leq n$, and furthermore establishes the asymptotics of this sum as $2\cdot3^{n}n^{-1}(1+o(1))$. This result is in line with our estimation: the central coefficient of $P_n(q)$ can be approximated by
$$\frac{P_{n,0}(1)}{\sqrt{2\pi g_P(n,1)}}=C\,3^nn^{-3/2}(1+o(1)).$$

We are in fact very optimistic that our analytic approach will have further implications. It seems that it is possible to adapt our approach for a proof of Conjectures~\ref{cjBorwein2} and \ref{cjBorwein3}. It remains to be seen whether these ideas may also finally lead to a full proof of Bressoud's Conjecture. Furthermore, we believe that they may also provide a basis for establishing open unimodality and log-concavity questions concerning polynomials given by products/quotients of factors of the form $1-q^k$, as found for example in \cite{MR2425713}. 

\section*{Acknowledgements}
We are indebted to Prof.~Manuel Kauers at JKU Linz for optimizing and running the verification program for $n\le 7000$.

\appendix

\section*{Appendix: Auxiliary inequalities}

\setcounter{equation}{0}
\setcounter{theorem}{0}
\global\def\thesection{\mbox{A}}
This appendix contains several auxiliary inequalities used in the course of the proof. As their proofs are tedious, we put them here so as not to disturb the flow of the argument in the main text.

\begin{lemma}\label{leComplexIneq}
For all $z=re^{i\theta}\in\C$ such that $0\leq r\leq1$ and $|\theta|\leq\frac13\frac{(-\log r)}{(1-r)}$, we have
\begin{align}
\left|\frac{1}{1+z+z^2}\right|&<1.002, \label{eqIneqJ1}\\
\left|\frac{(1-z)^2}{1+z+z^2}\right|&<1.005, \label{eqIneqJ2}\\
\left|\frac{(1-z^2)(1+7z+z^2)}{(1+z+z^2)^3}\right|&<\frac75, \label{eqIneqH}\\
\left|\frac{1+12z-12z^2-56z^3-12z^4+12z^5+z^6}{(1+z+z^2)^4}\right|&<\frac53. \label{eqIneqH2}
\end{align}
\end{lemma}
\begin{proof}
Let $S$ be the region
$$\{z=re^{i\theta}\in\C\mid\,0\le r\leq1,\ |\theta|\leq\frac13\frac{(-\log r)}{(1-r)}\}.$$
\begin{figure}[!htbp]
\includegraphics[width=0.4\textwidth]{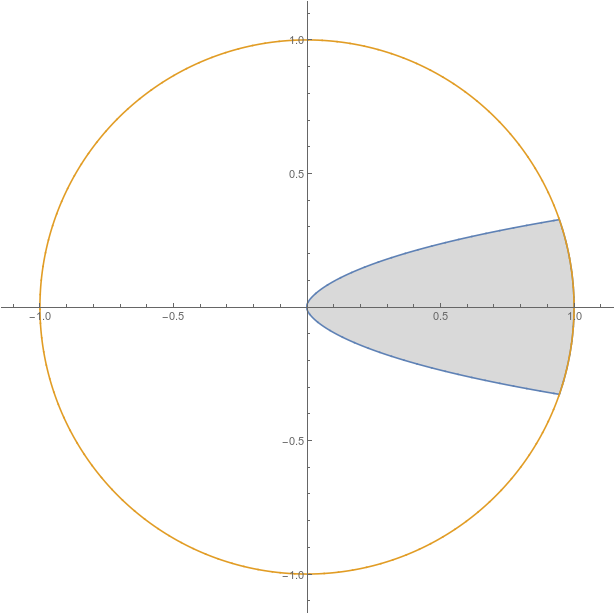}
\caption{Illustration of the region $S$ (shaded).}\label{fiRegion}
\end{figure}

All the rational functions on the left-hand sides of the inequalities are holomorphic on $S$, so the maximum modulus principle means it suffices to prove the inequalities on the boundary
$$\partial S=\left\{z=re^{i\theta}\in\C\,\middle|\,0<r<1,\ -\pi\le \theta\leq\pi,\ |\theta|=\frac13\frac{(-\log r)}{(1-r)}\right\}\cup\left\{e^{i\theta}\,\middle|\, |\theta|\leq\frac13\right\}.$$
The proof is done in a uniform way for all four rational functions (denoted by $f$ in the subsequent arguments): we choose $10^5$ uniformly distributed points $\{z_j\}_{1\leq j\leq1000}$ on $\partial S$, and argue that
\[
|f(z)|\leq\max_j|f(z_j)|+\left(\sup_{z\in \operatorname{conv}(S)}|f'(z)|\right)\left(\min_j|z-z_j|\right)
\]
holds for all $z\in\partial S$, where $\operatorname{conv}(S)$ denotes the convex hull of $S$. Then we evaluate $\max_j|f(z_j)|$ and use trivial upper bounds for $\sup_{z\in \operatorname{conv}(S)}|f'(z)|$ to complete the proof.
\end{proof}

\begin{lemma}\label{leIneqBeta}
Suppose that $u,v\in\R^+$. Then we have
\begin{equation}
\int^{\frac{3}{4}\frac{u}{v}}_0e^{-u x^2}\left(e^{v x^3}-1\right)\,dx<1.1\times\frac{v}{u^2}.
\end{equation}
\end{lemma}

\begin{proof}
By substituting $y=vx/u$ and $w=u^3/v^2$, the claimed inequality is equivalent to
\begin{equation}\label{eqIneqBizarre}
\int^{\frac34}_0we^{-w x^2}\left(e^{w x^3}-1\right)\,dx<1.1.
\end{equation}

Using a Taylor expansion of $e^{w x^3}$, we write the integral as a sum involving the lower incomplete gamma function $\gamma(s,a)=\int^a_0e^{-x}x^{s-1}dx$:
\begin{align*}
{}\phantom{=}\int^{\frac34}_0we^{-w x^2}\left(e^{w x^3}-1\right)\,dx
&=\sum_{k=1}^{\infty}\int^{\frac34}_0\frac{w^{k+1}}{k!}e^{-w x^2}x^{3k}\,dx\\
&=\sum_{k=1}^{\infty}\frac{1}{2k!\,w^{(k-1)/2}}\gamma\left(\frac{3k+1}{2},{\frac{9}{16}}w\right).
\end{align*}
We denote the summand by
\[
u(k,w):=\frac{1}{2k!\,w^{(k-1)/2}}\gamma\left(\frac{3k+1}{2},{\frac{9}{16}}w\right),
\]
and attempt to bound the summand from above.
\begin{itemize}
\item For $k=1$, $u(k,w)=u(1,w)$ can be bounded above by $\frac12\Gamma(2)=\frac12$.
\item For $k\geq2$, we first note that $\lim_{w\to0}u(k,w)=\lim_{w\to+\infty}u(k,w)=0$. This implies that the maximum value of $u(k,w)$ on $w\in\R^+$ occurs at a point where $\frac{\partial u(k,w)}{\partial w}=0$.

By taking the derivative, we see that any such point $w_0$ satisfies
\[
\gamma\left(\frac{3k+1}{2},{\frac{9}{16}}w_0\right)=\frac{2e^{-9w_0/16}\left(\frac34\sqrt{w}\right)^{3k+1}}{k-1}.
\]
By substituting this back into the expression for $u(k,w)$, we infer that
\begin{align*}
\sup_{w\geq0}u(k,w)&\leq \sup_{w\geq0}\frac{1}{2k!\,w^{(k-1)/2}}\frac{2e^{-9w/16}\left(\frac34\sqrt{w}\right)^{3k+1}}{k-1}\\
&=\sup_{w\geq0}\frac{e^{-9w/16}w^{k+1}\left(\frac34\right)^{3k+1}}{k!\,(k-1)}.
\end{align*}
Another derivative with respect to $w$ shows that this supremum occurs when $w=\frac{16}{9}(k+1)$, giving our final bound for $u(w,k)$:
\begin{align}
\sup_{w\geq0}u(k,w)&\leq \frac{\left(\frac34\right)^{k-1}(k+1)^{k+1}}{k!e^{k+1}(k-1)}\nonumber \\
\label{eqUUpperbound}&<\frac{\left(\frac34\right)^{k-1}\sqrt{k+1}}{(k-1)\sqrt{2\pi}},
\end{align}
where the last step used Stirling's approximation $n^n<\frac{n!e^n}{\sqrt{2\pi n}}$.
\end{itemize}

Directly using the upper bound \eqref{eqUUpperbound}, we get
\[
\sum_{k=1}^{\infty}u(k,w)<\frac12+\frac{1}{\sqrt{2\pi}}\sum_{k=2}^{\infty}\frac{\left(\frac34\right)^{k-1}\sqrt{k+1}}{(k-1)}\approx1.60608,
\]
which is worse than what we claimed. Instead, we use \eqref{eqUUpperbound} for the terms with $k>10$, and conclude that
\[
\sum_{k=11}^{\infty}u(k,w)<\frac{1}{\sqrt{2\pi}}\sum_{k=11}^{\infty}\frac{\left(\frac34\right)^{k-1}\sqrt{k+1}}{(k-1)}\approx0.027469<0.03.
\]

As for the leftover terms where $2\leq k\leq 10$, we first give a crude bound for large $w$ by noticing that
\[
u(k,w)=\frac{1}{2k!\,w^{(k-1)/2}}\gamma\left(\frac{3k+1}{2},{\frac{9}{16}}w\right)<\frac{1}{2k!\,w^{(k-1)/2}}\Gamma\left(\frac{3k+1}{2}\right).
\]
This inequality implies that if $w\geq25$, then we have
\[
\sum_{k=2}^{10}u(k,w)<\sum_{k=2}^{10}\frac{1}{2k!\,5^{k-1}}\Gamma\left(\frac{3k+1}{2}\right)\approx 0.4446.
\]
The interval $[0,25]$ is treated using the same method as in Lemma~\ref{leComplexIneq}, and the resulting upper bound is approximately $0.5677<0.57$.

Combining all the above arguments, we obtain
\[
\int^{\frac34}_0we^{-w x^2}\left(e^{w x^3}-1\right)\,dx=\sum_{k=1}^{\infty}u(k,w)<\frac12+0.03+0.57=1.1.
\qedhere
\]
\end{proof}

%{\color{red}INCOMPLETE YET}
%Making use of the series expansion
%\[
%\gamma(s,a)=\sum_{l=0}^{\infty}\frac{(-1)^la^{s+l}}{l!(s+l)},
%\]
%we arrive at another expression for the integral in \eqref{eqIneqBizarre},
%\[
%\begin{split}
%\int^1_0we^{-w x^2}\left(e^{w x^3}-1\right)\,dx&=\sum_{k=1}^{\infty}\sum_{l=1}^{\infty}\frac{(-1)^lw^{k+l+1}}{k!l!(2k+3l+1)}\\
%&=\sum_{n=1}^{\infty}(-1)^nw^{n+1}\sum_{k=1}^n\frac{1}{k!(n-k)!(2n+k+1)}.
%\end{split}
%\]
%The inner summation with respect to $k$ evaluates to
%\[
%\sum_{k=1}^n\frac{1}{k!(n-k)!(2n+k+1)}=\frac{(2n)!}{(3n+1)!}-\frac{1}{(2n+1)n!},
%\]
%
%which gives a closed form expression of the integral in \eqref{eqIneqBizarre}:
%\[
%\int^1_0we^{-w x^2}\left(e^{w x^3}-1\right)\,dx=w-\frac{\sqrt{\pi w}}{2}\erf\sqrt{w}-\frac{w^2}{12}{}_2 F_2\left(\begin{matrix}1&3/2\\5/3&7/3\end{matrix}\middle.;-\frac{4w}{27}\right).
%\]
%
%With this closed form, we calculate that the maximum of the function is approximately $2.3022$ at $w\approx15.3641$, thus concludes the proof.

\begin{lemma}\label{leIneqBeta2}
Suppose that $u,v\in\R^+$. Then we have
\begin{equation}
\int^{\frac{3}{4\sqrt2}\sqrt{\frac{u}{v}}}_0e^{-u x^2}\left(e^{v x^4}-1\right)\,dx<\frac{1}{3\sqrt3}\frac{v^{1/2}}{u^{3/2}}.
\end{equation}
\end{lemma}

\begin{proof}
By substituting $x=\sqrt{u/v}y$ and $w=u^2/v$, the claimed inequality is equivalent to
\begin{equation}\label{eqIneqBizarre2}
\int^{\frac{3}{4\sqrt2}}_0we^{-w x^2}\left(e^{w x^4}-1\right)\,dx<\frac{1}{3\sqrt3}.
\end{equation}

Using the Taylor expansion of $e^{w x^4}$, we write the integral as a sum involving the lower incomplete gamma function $\gamma(s,a)=\int^a_0e^{-x}x^{s-1}dx$,
\begin{align*}
\int^{\frac{3}{4\sqrt2}}_0we^{-w x^2}\left(e^{w x^4}-1\right)\,dx
&=\sum_{k=1}^{\infty}\int^{\frac{3}{4\sqrt2}}_0\frac{w^{k+1}}{k!}e^{-w x^2}x^{4k}\,dx\\
&=\sum_{k=1}^{\infty}\frac{1}{2k!w^{(2k-1)/2}}\gamma\left(\frac{4k+1}{2},{\frac{9}{32}}w\right).
\end{align*}

We denote the summand by
\[
u(k,w):=\frac{1}{2k!w^{(2k-1)/2}}\gamma\left(\frac{4k+1}{2},{\frac{9}{32}}w\right),
\]
and attempt to bound the summand from above.
We first note that
$$\lim_{w\to0}u(k,w)=\lim_{w\to+\infty}u(k,w)=0.$$
This means that the maximum value of $u(k,w)$ on $w\in\R^+$ occurs at a point where $\frac{\partial u(k,w)}{\partial w}=0$.

By taking a derivative, we can see any such point $w_0$ satisfies
\[
\gamma\left(\frac{4k+1}{2},{\frac{9}{32}}w_0\right)=\frac{2e^{-9w_0/32}\left(\frac{9}{32}w_0\right)^{(4k+1)/2}}{2k-1}.
\]
substituting it back into the expression of $u(k,w)$, we are able to infer that
\begin{align*}
\sup_{w\geq0}u(k,w)&\leq \sup_{w\geq0}\frac{1}{2k!w^{(k-1)/2}}\frac{2e^{-9w/32}\left(\frac{9}{32}w\right)^{(4k+1)/2}}{2k-1}\\
&=\sup_{w\geq0}\frac{e^{-9w/32}w^{k+1}\left(\frac{9}{32}\right)^{(4k+1)/2}}{k!(2k-1)}.
\end{align*}
Another derivative with respect to $w$ shows that this supremum occurs when $w=\frac{32}{9}(k+1)$, giving our final bound for $u(w,k)$:
\begin{align}
\sup_{w\geq0}u(k,w)&\leq \frac{\left(\frac{9}{32}\right)^{k-1/2}(k+1)^{k+1}}{k!e^{k+1}(2k-1)}\nonumber \\
\label{eqUUpperbound2}&<\frac{\left(\frac{9}{32}\right)^{k-1/2}\sqrt{k+1}}{(2k-1)\sqrt{2\pi}},
\end{align}
where the last step used Stirling's approximation $n^n<\frac{n!e^n}{\sqrt{2\pi n}}$.

Similar to the proof of Lemma~\ref{leIneqBeta}, we use \eqref{eqUUpperbound2} on the terms with $k\geq2$, and conclude that
\[
\sum_{k=2}^{\infty}u(k,w)<\frac{1}{\sqrt{2\pi}}\sum_{k=2}^{\infty}\frac{\left(\frac{9}{32}\right)^{k-1/2}\sqrt{k+1}}{(2k-1)}\approx0.04303<0.0431.
\]

As for the term $u(1,w)$, we first note that
\[
u(1,w)=\frac{\gamma(5/2,9w/32)}{2\sqrt{w}}<\frac{\Gamma(5/2)}{2\sqrt{w}}=\frac{3\sqrt\pi}{8\sqrt{w}},
\]
therefore $u(1,w)<\frac18$ if $w\geq9\pi$. The interval $[0,9\pi]$ is treated using the same method as in Lemma~\ref{leComplexIneq}, and the resulting upper bound is approximately $0.14875<0.1488$.

Combining all the above arguments, we conclude that
\[
\int^{\frac{3}{4\sqrt2}}_0we^{-w x^2}\left(e^{w x^4}-1\right)\,dx=\sum_{k=1}^{\infty}u(k,w)<0.1488+0.0431=0.1919<\frac{1}{3\sqrt3}.
\qedhere
\]
\end{proof}

The next lemma deals with inequalities between sums
of the form $\sum_{k=1}^{n}k^ar^k$.

\begin{lemma}
For all $n\in\Z^+$ and all $r\in(0,1]$, we have
\begin{align}
3r^2\left(\sum_{k=1}^{n}k^2r^k\right)&\geq \left(\sum_{k=1}^{n}r^k\right)^3,\label{eqIneqR2_000}\\
(r+1)\left(\sum_{k=1}^{n}r^k\right)^3&\geq r^2\left(\sum_{k=1}^{n}k^2r^k\right),\label{eqIneqR000_2}\\
(r^2+4r+1)\left(\sum_{k=1}^{n}r^k\right)\left(\sum_{k=1}^{n}k^2r^k\right)&\geq r(r+1)\left(\sum_{k=1}^{n}k^3r^k\right),\label{eqIneqR02_3}\\
(r^2+4r+1)^2\left(\sum_{k=1}^{n}k^2r^k\right)^3&\geq (r+1)^3\left(\sum_{k=1}^{n}k^3r^k\right)^2\left(\sum_{k=1}^{n}r^k\right),\label{eqIneqR222_033}\\
(r^2+10r+1)\left(\sum_{k=1}^{n}r^k\right)^2\left(\sum_{k=1}^{n}k^2r^k\right)&\geq r^2\left(\sum_{k=1}^{n}k^4r^k\right),\label{eqIneqR002_4}\\
(r^2+10r+1)\left(\sum_{k=1}^{n}k^2r^k\right)^2&\geq (r+1)\left(\sum_{k=1}^{n}k^4r^k\right)\left(\sum_{k=1}^{n}r^k\right).\label{eqIneqR22_04}
\end{align}
\end{lemma}
\begin{proof}
For simplicity of notation, we use $X_m$ to denote the sum $\sum_{k=1}^{n}k^mr^k$. The reader should observe that, for fixed~$m$, the sum $X_m$ can be evaluated into a rational function in~$r$ by applying the binomial theorem.

The first inequality is proved by noticing that the coefficient
$[r^k](3r^2X_2-X_0^3)$ is equal to\break $3(k-2)^2-\binom{k-1}{2}>0$ for $3\leq k\leq n+2$, and is negative for $n+3\leq k\leq 3n$. Moreover, the sum of the coefficients is equal to
\[
3\sum_{k=1}^{n}k^2-n^3=\frac{n(3n+1)}{2}>0.
\]
So we have
\[
3r^2\left(\sum_{k=1}^{n}k^2r^k\right)-\left(\sum_{k=1}^{n}r^k\right)^3\geq lr^{n+2}-lr^{n+3}+\frac{n(3n+1)}{2}r^{n+3}>0,
\]
where $l$ is the sum of all positive coefficients in $3r^2X_2-X_0^3$.

In order to prove the other inequalities, we give explicit formulas for the coefficients of the differences between both sides in those inequalities. More explicitly, after some tedious but routine calculations, we arrive at the following results:
\begin{itemize}
\item For \eqref{eqIneqR000_2}, we have
\[
(r+1)X_0^3-r^2X_2=r^{n+3}\sum_{k=0}^{2n-1}a_kr^k,
\]
where $a_k=(n+k+1)^2-3(k+1)^2$ for $0\leq k<n$, and $a_k=(2n-k-1)^2$ for $n\leq k<2n$.
\item For \eqref{eqIneqR02_3}, we have
\[
(r^2+4r+1)X_0X_2-r(r+1)X_3=r^{n+2}\sum_{k=0}^{n}b_kr^k,
\]
where $b_0=n(n+1)^2-1$, $b_n=n^2$, $b_k=n(n+1)(2n+1)-(2k+1)(k^2+k+1)$ for $0<k<n$.
\item For \eqref{eqIneqR222_033}, we have
\[
(r^2+4r+1)^2X_2^3-(r+1)^3X_0X_3^2=\frac{nr^{n+3}}{210}\sum_{k=0}^{2n+1}c_kr^k,
\]
where
\[
c_k=\begin{cases}
-12 k^7-42 k^6 n-84 k^6+168 k^5 n^2+126 k^5 n-294 k^5+420 k^4 n^2+315 k^4 n-630 k^4\kern-1.5cm\\
\kern.5cm
+1400 k^3 n^2+1680 k^3 n-798 k^3+1680 k^2 n^2+2247 k^2 n-546 k^2+1372 k n^2+1974 k n\kern-1.5cm\\
\kern.5cm
-156 k+420 n^2+630 n ,& 0\leq k<n, \\
114 n^7+462 n^6+1211 n^5+1470 n^4+301 n^3-252 n^2-156 n, & k=n, \\
12 k^7+42 k^6 n+84 k^6-168 k^5 n^2-126 k^5 n+294 k^5-420 k^4 n^2-315 k^4 n+630 k^4\kern-1.5cm\\
\kern.5cm
-840 k^3 n^4-1680 k^3 n^3-2240 k^3 n^2-1680 k^3 n+798 k^3+3276 k^2 n^5+7560 k^2 n^4\kern-1.5cm\\
\kern.5cm
+5250 k^2 n^3-420 k^2 n^2-1953 k^2 n+546 k^2-3024 k n^6-6468 k n^5-3360 k n^4+840 k n^3 \kern-1.5cm\\
\kern.5cm
-28 k n^2-1386 k n+156 k+816 n^7+1512 n^6+1372 n^5+2100 n^4+2114 n^3\kern-1.5cm\\
\kern.5cm
+588 n^2-312 n,& n<k\leq2n, \\
210 n^5, & k=2 n+1.
\end{cases}
\]
\item For \eqref{eqIneqR002_4}, we have
\[
(r^2+10r+1)X_0^2X_2-r^2X_4=r^{n+3}\sum_{k=0}^{2n-1}d_kr^k,
\]
where $d_k=n^4+4(k+1)n^3-2(k+1)^4+(6k+5)n^2+2(k+1)n$ for $0\leq k<n$, and $d_k=(2n-1-k)^2(2n^2+(k+1)^2)+2n(3n+1)(2n-1-k)+n^2$ for $n\leq k<2n$.
\item For \eqref{eqIneqR22_04}, we have
\[
(r^2+10r+1)X_2^2-(r+1)X_0X_4=nr^{n+2}\sum_{k=0}^{n}e_kr^k,
\]
where $e_k=n^3+2(n-k)(kn^2+(3k^2+4k+2)n+(k+1)^2(k+2))$.
\end{itemize}
It is easy to see that $a_k$, $b_k$, $c_n$, $c_{2n+1}$, $d_k$ and $e_k$ are non-negative. For the remaining $c_k$'s, we distinguish two cases:
\begin{itemize}
\item $0\leq k<n$. Here we substitute $k=\lambda n$ with $0\leq\lambda<1$ to see that
\begin{multline*}
c_k=\left(-12 \lambda ^7-42 \lambda ^6+168 \lambda ^5\right) n^7+\left(-84 \lambda ^6+126 \lambda ^5+420 \lambda ^4\right)
   n^6+\left(-294 \lambda ^5+315 \lambda ^4+1400 \lambda ^3\right) n^5\\
+\left(-630 \lambda ^4+1680 \lambda ^3+1680 \lambda
   ^2\right) n^4+\left(-798 \lambda ^3+2247 \lambda ^2+1372 \lambda \right) n^3+\left(-546 \lambda ^2+1974 \lambda
   +420\right) n^2\\
+(630-156 \lambda ) n,
\end{multline*}
and note that $0\leq\lambda<1$ implies that every coefficient above is non-negative.
\item $n\leq k\leq2n$. Similarly, we substitute $k=(2-\lambda)n$ to write
\begin{multline*}
c_k=\left(-12 \lambda ^7+210 \lambda ^6-1344 \lambda ^5+4200 \lambda ^4-5880 \lambda ^3+2940 \lambda ^2\right) n^7\\
+\left(84 \lambda ^6-882 \lambda ^5+3360 \lambda ^4-3360 \lambda ^3-2520 \lambda ^2+3780 \lambda \right) n^6\\
   +\left(-294 \lambda ^5+2625 \lambda ^4-7000 \lambda ^3+7770 \lambda ^2-4200 \lambda +2100\right) n^5\\
   +\left(630 \lambda ^4-3360 \lambda ^3+4620 \lambda ^2+840 \lambda -1260\right) n^4+\left(-798 \lambda ^3+2835 \lambda ^2-1736 \lambda +630\right)
   n^3
\\+\left(546 \lambda ^2-798 \lambda \right) n^2-156 \lambda  n.
\end{multline*}
In this case some of the coefficients (namely, the coefficients of $n$, $n^2$ and of $n^4$) are negative. However, by exploiting the fact that $n\geq1$ and that
\begin{align*}
[n^5]c_k+[n^4]c_k&=-294\lambda^5+3255\lambda^4-10360\lambda^3+12390\lambda^2-3360\lambda+840\\
&=840(1-2\lambda+2\lambda^2)^2+7\lambda^2(810-520\lambda-15\lambda^2-42\lambda^3)>0,\\
[n^3]c_k+[n^2]c_k+[n^1]c_k&=-798\lambda^3+3381\lambda^2-2690\lambda+630\\
&=523(1-2\lambda)^2+(1-\lambda)(798\lambda^2-491\lambda+107)>0,
\end{align*}
we can still directly conclude that $c_k\geq0$.
\qedhere
\end{itemize}
\end{proof}

\begin{lemma}\label{leIneqRB1}
Let $r\in\R^+$ and $b\geq2$. Then we have
\[
\frac{(1-r^{b+1})(1-r^{b-1})}{r^{b-1}(b^2-1)(1-r)^2}\geq\frac{(1+r^{b/3})(1+r^b)}{2(r^{b/3}+r^b)}.
\]
\end{lemma}

\begin{proof}
Let $z=\frac{b}{2}\log r$. The lemma is equivalent to
\begin{equation}\label{eq11}
\frac{\sinh(z+z/b)\sinh(z-z/b)}{(b^2-1)\sinh^2(z/b)}\geq\frac{\cosh(z/3)\cosh z}{\cosh(2z/3)}.
\end{equation}

When $b=2$, the difference between the two sides of \eqref{eq11} is
\begin{align*}
\frac{\sinh(3z/2)}{3\sinh(z/2)}&-{}\frac{\cosh(z/3)\cosh z}{\cosh(2z/3)}\\
&=\frac{1}{3\cosh(2z/3)}\left(\cosh(2z/3)(1+2\cosh z)-\cosh(z/3)\cosh z\right)\\
&=\frac{(\cosh(z/3)-1)^2(2\cosh(z/3)+1)(8\cosh^2(z/3)+6\cosh(z/3)-1)}{3\cosh(2z/3)}\geq0.
\end{align*}
We now proceed to prove that the left-hand side of \eqref{eq11} is increasing with respect to $b$.

To this end, we calculate the derivative
\begin{multline*}
\frac{\partial}{\partial b}\frac{\sinh(z+z/b)\sinh(z-z/b)}{(b^2-1)\sinh^2(z/b)}\\
=2\frac{(b^2-1)z\cosh(z/b)\sinh^2z-b^3\sinh(z/b)\sinh(z-z/b)\sinh(z+z/b)}{b^2(b^2-1)^2\sinh^3(z/b)}.
\end{multline*}
So it suffices to prove that
\[
(b^2-1)z\cosh(z/b)\sinh^2z\cosh(z/b)\sinh^2(z)\geq b^3\sinh(z/b)\sinh(z-z/b)\sinh(z+z/b),
\]
or equivalently,
\[
\frac{\sinh(2z/b)\sinh^2(z)}{(2z/b)z^2}\geq \frac{\sinh(z/b)^2\sinh(z-z/b)\sinh(z+z/b)}{(z/b)^2(z-z/b)(z+z/b)}.
\]
Taking the logarithm on both sides, and defining $f(x):=\log\frac{\sinh x}{x}$ and $f(0):=0$, we arrive at another equivalent form,
\[
f(z+z/b)+f(z-z/b)-2f(z)+2f(z/b)-f(2z/b)-f(0)\leq0.
\]
The left-hand side can be written as a triple integral,
\[
f(z+z/b)+f(z-z/b)-2f(z)+2f(z/b)-f(2z/b)-f(0)=\iiint_{[0,z/b]^2\times[z/b,z]}f'''(\gamma+\alpha-\beta)d\alpha\,d\beta\,d\gamma,
\]
and we conclude the proof by noting that
\[
f'''(x)=2(\cosh x(\sinh x)^{-3}-x^{-3})\leq0.
\qedhere
\]
\end{proof}

The following two inequalities are used in the proof of Theorem~\ref{thMain}.
\begin{lemma}\label{leRDerivative1}
Suppose that $0<r\leq1$, and $n\geq1$. Then we have
\[
\frac{(1-r^n)^2}{(1-r)^2}\geq\frac{\partial}{\partial r}\frac{1-r^n}{1-r}\geq\frac{(1-r^n)(1-r^{(n-1)/2})}{(1-r)^2}.
\]
\end{lemma}
\begin{proof}
Direct calculation reveals that
\begin{align*}
\frac{\partial}{\partial r}\frac{1-r^n}{1-r}-\frac{(1-r^n)(1-r^{(n-1)/2})}{(1-r)^2}&=\frac{r^{n-1/2}}{(1-r)^2}\left(r^{-n/2}-r^{n/2}-n(r^{-1/2}-r^{1/2})\right),\\
\frac{(1-r^n)^2}{(1-r)^2}-\frac{\partial}{\partial r}\frac{1-r^n}{1-r}&=\frac{r^{n-1}}{(1-r)^2}\left(n(1-r)-r(1-r^n)\right).
\end{align*}
Therefore, the lemma follows from the elementary inequality
\[
\frac{r(1-r^n)}{1-r}\leq n\leq \frac{r^{-n/2}-r^{n/2}}{r^{-1/2}-r^{1/2}}.
\qedhere
\]
\end{proof}

\begin{lemma}\label{leRRational1}
Suppose that $n\geq 6924$, and $r\in(\exp(-\sqrt{\alpha/n}),1]$ with $\alpha=2/\sqrt3$. Then we have
\[
\frac{(1-r^{n/12})(1-r^{(n-12)/24})}{(1-r)(1-r^n)}\geq24.
\]
\end{lemma}
\begin{proof}
First of all, the condition $n\geq6924$ implies that $1-r^{(n-12)/24}\geq1-r^{288}$, as well as $r>\exp(-\sqrt{\alpha/6924})>e^{-1/72}$.

Noting that the function $\frac{1-r^{n/12}}{1-r^n}$ is increasing with respect to $n$, we conclude that
\begin{align*}
\frac{(1-r^{n/12})(1-r^{(n-12)/24})}{(1-r)(1-r^n)}&\geq\frac{(1-r^{n/12})(1-r^{288})}{(1-r)(1-r^n)}\geq\frac{(1-r^{48})(1-r^{288})}{(1-r)(1-r^{576})}=\frac{1-r^{48}}{(1-r)(1+r^{288})}.
\end{align*}
Thus it suffices to prove that $1-r^{48}\geq24(1-r)(1+r^{288})$ for $r\in(e^{-1/72},1]$. To this end, we prove that $1-r^{48}-24(1-r)(1+r^{288})$ is decreasing on $(e^{-1/72},1]$ by calculating the derivative. We have
\begin{align*}
\frac{d}{dr}\left((1-r^{48})-24(1-r)(1+r^{288})\right)&=24(1-2r^{48}+r^{288})-48(1-r)(r^{47}+144r^{287})\\
\leq 24r^{48}(r^{-48}+r^{240}-2)&\leq 24r^{48}\max\left(e^{2/3}+e^{-10/3}-2,1+1-2\right)=0,
\end{align*}
where we exploit the convexity of the function $r\mapsto r^{-48}+r^{240}-2$.
\end{proof}

The final inequality in the appendix gives a simple rational lower bound for the Chebyshev polynomials of the first kind $T_n(x)$, defined by $T_n(\cos\theta)=\cos n\theta$.
\begin{lemma}\label{leIneqChebyshev}
For all $x\in[-1,1]$ and all $n\in\Z^+$, we have
\[
T_n(x)\geq \frac{-n^2(1-x)(2x+3)+3(1+x)}{n^2(1-x)+3(1+x)}.
\]
\end{lemma}
\begin{proof}
If $n=1$, then both sides are equal to $x$. From now on we assume $n\geq2$.

If $-1\leq x\leq 1-\frac{3}{n^2}$, then we have
\[
\frac{-n^2(1-x)(2x+3)+3(1+x)}{n^2(1-x)+3(1+x)}\leq-1\leq T_n(x).
\]
If $1-\frac{3}{n^2}\leq x\leq1$, then we write $\theta=\frac{n}{2}\arccos x$, so that
\[
0\leq\theta\leq n\arcsin\sqrt{\frac{3}{2n^2}}.
\]
The two sides of the inequalities can be rewritten as
\begin{equation*}
T_n(x)=\cos2\theta=1-2\sin^2\theta
\end{equation*}
and
\begin{align*}
\frac{-n^2(1-x)(2x+3)+3(1+x)}{n^2(1-x)+3(1+x)}&=1-\frac{n^2(2x+4)}{n^2+3\frac{1+x}{1-x}}\\
&=1-\frac{2\cos(2\theta/n)+4}{1+3n^{-2}\cot^2(\theta/n)}\\
&=1-\frac{6-4\sin^2(\theta/n)}{1+3n^{-2}\cot^2(\theta/n)}.
\end{align*}
Thus it suffices to prove that
\[
\sin^2\theta(1+3n^{-2}\cot^2(\theta/n))\leq3-2\sin^2(\theta/n)
\]
for all $n\geq2$ and $\theta\in[0,n\arcsin\sqrt{\frac{3}{2n^2}}]$.

For the last inequality, we make use of the well-known inequalities
\begin{align*}
\sin^2 x&\leq x^2-\frac{x^4}{3}+\frac{2x^6}{45},\\
\cot x&\leq\frac1x-\frac{x}{3}-\frac{x^3}{45},
\end{align*}
to conclude that
\begin{align*}
3-2\sin^2(\theta/n)&-{}\sin^2\theta(1+3n^{-2}\cot^2(\theta/n))\\
&\geq3-2\left(\frac{\theta^2}{n^2}-\frac{\theta^4}{3n^4}+\frac{2\theta^6}{45n^6}\right)-\left(\theta^2-\frac{\theta^4}{3}+\frac{2\theta^6}{45}\right)\left(1+\frac{3}{n^2}\left(\frac{n}{\theta}-\frac{\theta}{3n}-\frac{\theta^3}{45n^2}\right)^2\right)\\
&=\theta^4\left(\frac{(n^2-1)(7n^2-3)}{15n^4}-\frac{(n^2-2)(2n^4-3)}{45n^6}\theta^2-\frac{6n^4-10n^2+1}{675n^8}\theta^4\right.\\
&\kern2cm\left.-\frac{4n^2-1}{2025n^8}\theta^6-\frac{2}{30375n^8}\theta^8\right).
\end{align*}
The last factor is clearly decreasing with respect to $\theta$ when $n\geq2$, so we proceed to find an upper bound for $\theta$. We note that
\[
\frac{d}{dn}n\arcsin\sqrt{\frac{3}{2n^2}}=\arcsin\sqrt{\frac{3}{2n^2}}-\frac{1}{\sqrt{\frac{2n^2}{3}-1}},
\]
and after substituting $\phi=\arcsin\sqrt{\frac{3}{2n^2}}$ we see that the derivative is equal to $\phi-\tan\phi<0$. So $n\arcsin\sqrt{\frac{3}{2n^2}}$ is decreasing with respect to $n$. This implies that we always have
\[
0\leq\theta\leq n\arcsin\sqrt{\frac{3}{2n^2}}\leq 2\arcsin\sqrt\frac38<\sqrt{\frac{15}{8}}.
\]
Using this bound, we conclude that
\begin{align*}
3-2\sin^2(\theta/n)&-{}\sin^2\theta(1+3n^{-2}\cot^2(\theta/n))\\
&\geq\theta^4\left(\frac{(n^2-1)(7n^2-3)}{15n^4}-\frac{(n^2-2)(2n^4-3)}{45n^6}\theta^2-\frac{6n^4-10n^2+1}{675n^8}\theta^4\right.\\
&\kern2cm\left.-\frac{4n^2-1}{2025n^8}\theta^6-\frac{2}{30375n^8}\theta^8\right)\\
&\geq\theta^4\left(\frac{(n^2-1)(7n^2-3)}{15n^4}-\frac{(n^2-2)(2n^4-3)}{45n^6}\frac{15}{8}-\frac{6n^4-10n^2+1}{675n^8}\left(\frac{15}{8}\right)^2\right.\\
&\kern2cm\left.-\frac{4n^2-1}{2025n^8}\left(\frac{15}{8}\right)^3-\frac{2}{30375n^8}\left(\frac{15}{8}\right)^4\right)\\
&=\frac{15\theta^4}{2^{17}n^8}\left(3584n^8-15360n^6+17216n^4-6480n^2-85\right)\\
&=\frac{15\theta^4}{2^{17}n^8}\left(512n^4(n^2-4)(7n^2-2)+6480(n^2-1)(2n^2+1)+160n^4+6395\right)\\
&\geq0.\qedhere
\end{align*}
\end{proof}

\bibliographystyle{siam}
%\nocite{*}
\bibliography{borwein}

\end{document}